\documentclass{amsart}

\usepackage[english]{babel}
\usepackage[utf8]{inputenc}
\usepackage{xcolor}

\usepackage{enumitem}
\usepackage{amsmath}
\usepackage{amsfonts}
\usepackage{amssymb}
\usepackage{mathrsfs}
\usepackage{mathtools}
\usepackage{caption}
\usepackage[hidelinks]{hyperref}
\usepackage{amsthm}
\usepackage[norefs, nocites]{refcheck}
\newtheorem{theorem}{Theorem}[section]
\newtheorem{lemma}[theorem]{Lemma}
\newtheorem{proposition}[theorem]{Proposition}
\newtheorem{corollary}[theorem]{Corollary}

\theoremstyle{definition}
\newtheorem{example}[theorem]{Example}
\newtheorem{remark}[theorem]{Remark}

\author{Bart Michels}
\title{The geometry of maximal flat submanifolds of symmetric spaces}
\thanks{Universit\'{e} Sorbonne Paris Nord, LAGA, CNRS, UMR 7539,  F-93430, Villetaneuse, France. Email: \texttt{michels@math.univ-paris13.fr}}

\begin{document}

\begin{abstract}Motivated by spectral asymptotics for orbital integrals in a relative trace formula, we generalize a number of geometric properties of geodesics in the hyperbolic plane, to maximal flat submanifolds of symmetric spaces of non-compact type.
\end{abstract}

\maketitle

\section{Introduction}

Let $G$ be a connected semisimple Lie group with finite center. Let $G = NAK$ be a choice of Iwasawa decomposition. Define $\mathfrak a = \operatorname{Lie}(A)$, and recall Harish-Chandra's formula for the spherical function of parameter $i\lambda \in i\mathfrak a^*$:
\begin{equation}\label{harishchandraformula}
\varphi_{i\lambda}(g) = \int_K e^{(\rho+i\lambda)(H(kg))} dk \,,
\end{equation}
where $H : G \to \mathfrak a$ is the Iwasawa projection and $\rho \in \mathfrak a^*$ the half-sum of positive roots. In \cite{Michels2022} we are interested in spectral asymptotics in a relative trace formula for maximal flat submanifolds of an associated locally symmetric space. In the analysis one requires asymptotics for the integral
\[ \int_A \varphi_{i\lambda}(a) b(a) da \]
(with a smooth cutoff $b(a)$)\,, as well as bounds for twisted versions thereof, as $\lambda$ grows in $\mathfrak a^*$. The problems are relative analogues of the problem of bounding spherical functions considered in \cite{duistermaat1983}.

In view of \eqref{harishchandraformula}, one is naturally led to study the behavior of the Iwasawa $\mathfrak a$-projection along sets of the form $kA$. More specifically, we are interested in the critical points of
\[ \lambda(H(ka)) \]
as a function of $a \in A$.
The maximal flat submanifolds of the symmetric space $G/K$ are the images of the sets $gA$. When $G = \operatorname{PSL}_2(\mathbb R)$, then $G/K$ is the hyperbolic plane, the maximal flats are the geodesics, and the motivating problem of bounding orbital integrals is solved in \cite{Marshall2016}. It takes as input classical facts about the geometry of geodesics in the hyperbolic plane. For general semisimple $G$ we require analogous results for maximal flat submanifolds, which to our knowledge are not established. This article is concerned with them.

The article is divided into three parts, which are about the Iwasawa $N$-, $A$- and $K$- projections of the sets $gA$, in each of which we generalize geometric facts that are apparent in the case of $\operatorname{PSL}_2(\mathbb R)$.

\subsection{Statement of results}

Consider the hyperbolic plane $\mathbb H = \{x+iy  \in \mathbb C : y > 0\}$ with its hyperbolic metric.
The group $\operatorname{PSL}_2(\mathbb R)$ acts on $\mathbb H$ by homographies. We make the standard choice of Iwasawa decomposition: Let $N$ be the subgroup of unipotent upper triangular matrices, $A$ the subgroup of diagonal matrices and $K =  \operatorname{PSO}_2(\mathbb R)$. Then $\operatorname{PSL}_2(\mathbb R) \cong N \times A \times K$, the diffeomorphism being given by multiplication. The upper half-plane model naturally lends itself to describe the Iwasawa projections geometrically. The element
\[ g = \begin{pmatrix}
1 & x \\
0 & 1
\end{pmatrix}\begin{pmatrix}
\sqrt y & 0 \\
0 & 1/\sqrt{y}
\end{pmatrix} \in NA\] sends $i$ to $x+iy$. That is, the real part of $gi$ can be identified with the $N$-projection of $g$, and the imaginary part with the $A$-projection.

By a (maximal) geodesic in $\mathbf H$ we mean the $1$-dimensional submanifold defined by it, although we will sometimes informally speak of geodesics with an orientation. The geodesics in $\mathbb H$ are the semicircles with centers on the horizontal axis, together with the vertical lines. The action of $\operatorname{PSL}_2(\mathbb R)$ on $i$ induces a diffeomorphism between $A$ and the vertical geodesic through $i$. Every geodesic is of the form $gAi$ with $g \in \operatorname{PSL}_2(\mathbb R)$.

Observe that the real part of every geodesic is a bounded set. We generalize this observation to semisimple Lie groups $G$, as follows.

\begin{theorem}\label{npartbounded}Let $G$ be a connected semisimple Lie group. For all $g \in G$, the $N$-projection of $gA$ is a relatively compact set.
\end{theorem}

Theorem~\ref{npartbounded} is proved in \S\ref{secNprojection} by diving into the mechanics of the Gram--Schmidt process, and showing that the orthogonalization part can be done with uniformly bounded operations. In fact, we will prove a stronger version with some uniformity, which is Theorem~\ref{npartboundeduniform}. The uniform version requires to partition the set of all maximal flats, which is naturally identified with $G/N_G(A)$, into a Zariski open `generic' set and several lower-dimensional `exceptional' sets. This partition generalizes the distinction between semicircles and vertical geodesics in the upper half plane $\mathbb H$, in which case the semicircles form the generic set. The partition is not dependent on any choice of model for the symmetric space $G/K$, but inherent to the choice of Iwasawa decomposition of $G$. Some of the lower-dimensional sets come from semistandard Levi subgroups, and it is no surprise that they are exceptional. But in general there are other exceptional sets, and the partition remains quite mysterious.

The second group of results concerns the Iwasawa $A$-projection. In the case of $G = \operatorname{PSL}_2(\mathbb R)$, define the height of a point in $\mathbb H$ to be its imaginary part. The heights of the points of a geodesic $gAi$ form a bounded-from-above set precisely when the geodesic is a semicircle. In that case, the height is maximized at a unique critical point, which is the midpoint of the semicircle, and the height tends to $0$ at infinity on the geodesic. Such a critical point exists if and only if $gAi$ is not vertical, meaning that $g \notin N \cdot N_G(A)$.

For a general connected semisimple Lie group $G$, we prove the following.

\begin{theorem}\label{aresultsallintro}Let $\lambda \in \mathfrak a^*$ be an element that is positive with respect to the choice of Iwasawa decomposition, non-singular and which does not lie in a proper subspace spanned by roots. 
\begin{enumerate}[label = (\roman*)]
\item For all $g \in G$ the ``height function''
\begin{align*}
h_{\lambda, g} : A & \to \mathbb R \\
a &\mapsto \lambda(H(ga))
\end{align*}
has at most one critical point. If it exists, it is non-degenerate and maximizes $h_{\lambda, g}$. 
\item The set of $g \in G$ for which a given $a \in A$ is a critical point of $h_{\lambda, g}$, is a non-empty smooth submanifold of codimension $\dim(A)$.
The set of $g \in G$ for which $h_{\lambda, g}$ has a critical point, is open.
\end{enumerate}
\end{theorem}

We prove Theorem~\ref{aresultsallintro} in \S\ref{secAprojection}. The proof is quite technical and occupies the larger part of this article. Regarding the second part, note that it is by no means obvious that $h_{\lambda, g}$ has a critical point for even a single $g$, because the domain $A$ is noncompact. The way in which we prove existence is by varying $g \in K$, and realizing the elements $g$ with a given critical point as the minima of a certain smooth function on the compact group $K$. It is likely that the set of $g \in G$ for which $h_{\lambda, g}$ has a critical point is in fact dense and that this can be proved using a very different argument, which is related to Theorem~\ref{npartbounded}; see Remark~\ref{atonegativechamber}.

\begin{remark}
The critical points in Theorem~\ref{aresultsallintro} can be thought of as giving the midpoints of the flat $gA \subset S$. It is in general too much to hope that $H(ga)$ has a critical point as a function of $a$. That is, the critical points can depend on $\lambda \in \mathfrak a^*$. (See Example~\ref{examplemidpointdepends}.) We will not use the term ``midpoint'', all the more because we have found no way to generalize the observation that for $\operatorname{PSL}_2(\mathbb R)$, the critical point corresponds to the center of the semicircle $gAi \subset \mathbb H$.\end{remark}

\begin{remark}In Theorem~\ref{aresultsallintro}, many things break down when $\lambda$ is singular, lies in a proper subspace spanned by roots or is nonpositive: the nondegeneracy, uniqueness, and existence of critical points. Regarding non-degeneracy, see Remark~\ref{hessiancriticalsometimesdef}. For non-uniqueness and non-existence, see \S\ref{sectionexistence}.\end{remark}

Finally, we turn our attention to the Iwasawa $K$-projection. The Lie group $\operatorname{PSL}_2(\mathbb R)$ is naturally identified with the unit tangent bundle $T^1(\mathbb H)$ via its action on the vertical vector at the base point, $(i, (0, 1))$. The Iwasawa $K$-projection of an element $g$ corresponds to a choice of direction at the point $gi$. The elements of the Weyl group $N_K(A)$ correspond to the vertical directions, up and down. As we approach infinity along a geodesic, the tangent line to the geodesic tends to a vertical one.

We generalize this observation as follows.

\begin{theorem}\label{KparttoMcentralizer}Let $G$ be a connected semisimple Lie group. Let $g \in G$ and $H \in \mathfrak{a}$. Then the $K$-projection of $g e^{tH}$ tends to $N_K(A) Z_{K}(H)$ as $t \to + \infty$.
\end{theorem}

Theorem~\ref{KparttoMcentralizer} is proved by projecting the $K$-projection of the geodesic flow down to the Lie algebra in a specific way and realizing it as the flow of a vector field. The resulting dynamical system is quite mysterious, but the asymptotic behavior of individual orbits can be well understood. It might seem that Theorem~\ref{KparttoMcentralizer} is a statement about individual geodesics rather than maximal flats, but it is possible to formulate statements with uniformity in the variables $g$ and $H$; see Remark~\ref{remarkKvsNproofs}.

In the end the results on the $N$- and $K$-projections were not needed for the analysis in \cite{Michels2022}. But they complete the picture nicely, might be useful elsewhere, and they barely fail to provide alternative proofs of parts of Theorem~\ref{aresultsallintro}; see Remark~\ref{atonegativechamber}. A number of things remain mysterious, in particular the apparent relation between the partition in Theorem~\ref{npartboundeduniform} on the $N$-projection, and the dynamical system used in the proof of Theorem~\ref{KparttoMcentralizer}; see Remark~\ref{remarkKvsNproofs}.

\section{Preliminaries}

\label{preliminariesLiegroups}

\subsection{Lie groups and Lie algebras}

\label{notationliegroups}

Let $G$ be a reductive Lie group in the sense of Harish-Chandra \cite{harishchandra1975}; see also \cite[Chapter VII]{Knapp2002}. Most of the time $G$ will be connected semisimple with finite center and we will then simply say $G$ is semisimple. Let $K \subset G$ be a maximal compact subgroup and $\theta$ an involution of $G$ whose fixed point set is $K$. It induces an involution $\theta$ of the Lie algebra $\mathfrak g$, whose $+1$ and $-1$ eigenspaces we denote by $\mathfrak k$ and $\mathfrak p$ respectively. We denote the exponential of $X \in \mathfrak g$ by $\exp(X)$. Define the semisimple part $\mathfrak{g}_{ss} = [\mathfrak g, \mathfrak g]$. We extend the Killing form on $\mathfrak{g}_{ss}$ to a nondegenerate symmetric bilinear form $\langle \cdot, \cdot \rangle$ on $\mathfrak{g}$ that is positive definite on $\mathfrak{k}$ and negative definite on $\mathfrak{p}$, and with respect to which the center $\mathfrak{z}(\mathfrak{g})$ is orthogonal to $\mathfrak{g}_{ss}$. Define $\langle \cdot, \cdot \rangle_{\theta} = \langle \cdot, - \theta(\cdot) \rangle$, a positive definite symmetric bilinear form. All statements on $\mathfrak{g}$ involving norms, orthogonality and adjoints will be with respect to $\langle \cdot, \cdot \rangle_{\theta}$. Let $\mathfrak a \subset \mathfrak p$ be a maximal abelian subalgebra and $A = \exp(\mathfrak a)$. The choices of $\mathfrak a$ are all conjugate under $K$. Define $P = \exp(\mathfrak p)$. The multiplication map $P  \times K \to G$ is a diffeomorphism, known as the Cartan decomposition.

\subsection{Symmetric spaces and maximal flats}

\label{propsmaximalflatssection}

References for the following facts about symmetric spaces are \cite{Eberlein1996, Helgason1978}. 

Assume here that $G$ is semisimple. Then $\theta$ is a Cartan involution. The quotient $S = G/K$ carries a left-$G$-invariant Riemannian metric induced by the Killing form on $\mathfrak p$. It is a symmetric space of non-compact type, and every such space arises in this way.

The maximal flat submanifolds of $S$ are of the form $gAK$ with $g\in G$. Such $g$ is uniquely determined by the submanifold up to multiplication on the right by $N_G(A)$. When $\dim(A) = 1$, the maximal flats are precisely the geodesics. The rank of $G$ is defined to be $\dim(A)$.

Let $p \in P$. The tangent space $T_{pK}S$ is identified with $\mathfrak p$, using left multiplication by $p^{-1}$. Take $X \in \mathfrak p$ with norm $1$. The geodesic through $pK \in S$ with tangent vector $X$ has equation $t \mapsto p e^{tX}K$. A geodesic is called regular when a nonzero tangent vector at any (and hence every) point is a regular element of $\mathfrak g$. It is called singular otherwise. A geodesic is regular if and only if it lies in a unique maximal flat.

\subsection{Iwasawa decomposition}

\label{secroots}

Let $\Sigma$ be the set of restricted roots of $\mathfrak{a}$ in $\mathfrak{g}$. By convention, $0  \notin \Sigma$. We denote by $\mathfrak{g}_{\alpha}$ the root space of a root $\alpha \in \Sigma$, by $m(\alpha) = \dim(\mathfrak g_\alpha)$ its multiplicity and by $H_\alpha \in \mathfrak a$ the element corresponding to $\alpha$ under the isomorphism $\mathfrak a \cong \mathfrak a^*$ given by $\langle \cdot, \cdot \rangle$. Fix a set of positive roots $\Sigma^{+}$ with basis $\Pi$. Let $\mathfrak n \oplus \mathfrak a \oplus \mathfrak k$ and $N \times A \times K$ be the corresponding Iwasawa decompositions of $\mathfrak g$ and $G$. Define $M = Z_{K}(A)$ and $M' = N_{K}(A)$ and denote by $\mathfrak{m}$ the Lie algebra of $M$.

Denote the Lie algebra Iwasawa projections by $E_{\mathfrak{n}}$, $E_{\mathfrak{a}}$ and $E_{\mathfrak{k}}$. We have the orthogonal restricted root space decomposition
\begin{equation}\label{rootspacedecomp}
\mathfrak{g} = \mathfrak{a} \oplus \mathfrak{m} \oplus \bigoplus_{\alpha \in \Sigma} \mathfrak{g}_{\alpha}  \,.
\end{equation}
Denote the projection onto $\mathfrak g_{\alpha}$ by $R_\alpha$.

Denote the Iwasawa projections from $G$ onto $N$, $A$ and $K$ by $n$, $a$ and $\kappa$. Define the height $H(g) = \log(a(g)) \in \mathfrak{a}$, the logarithm being the Lie logarithm on $A$.

All choices of the data $(K, A, N)$ are conjugate by an element of $G$. When $G = \operatorname{GL}_n(\mathbb R)$ or $\operatorname{SL}_n(\mathbb R)$ we make all the standard choices: $K = \operatorname{O}_n(\mathbb R)$ respectively $\operatorname{SO}_n(\mathbb R)$, $A$ is the connected component of the diagonal subgroup and $N$ is the upper triangular unipotent subgroup.

\subsection{Centralizers}
\label{defgenericset}

\label{definitionlevis}
Denote by $\mathcal{L}$ the set of centralizers in $G$ of subgroups of $A$. They are the standard Levi subgroups of semistandard parabolic subgroups of $G$. We will denote such a centralizer typically by $L$. It is again reductive and inherits all the data as in the beginning of \S\ref{notationliegroups} and \S\ref{secroots} from $G$ in the natural way. We allow $G$ to be reductive because we will occasionally need to apply results to Levis $L \in \mathcal{L}$. When $L \in \mathcal{L}$ with Lie algebra $\mathfrak{l}$, define $\mathfrak{a}^{L} = \mathfrak{l}_{ss} \cap \mathfrak{a}$ and $\mathfrak{a}_{L} = \mathfrak{z}(\mathfrak{l}) \cap \mathfrak{a}$. Then $\mathfrak{a} = \mathfrak{a}^{L} \oplus \mathfrak{a}_{L}$ orthogonally. 
The set $\mathcal L$ contains $A$ and $G$, and when $G$ is semisimple we have $\mathfrak a^A = \mathfrak a_G = 0$ and $\mathfrak a_A = \mathfrak a^G = \mathfrak a$.

Define the positive Weyl chamber $\mathfrak{a}^{+} = \{ H \in \mathfrak{a} : \forall \alpha \in \Sigma^{+} : \alpha(H) > 0 \}$, the regular set
\[ \mathfrak{a}^{\operatorname{reg}} = \mathfrak{a} - \bigcup_{L \neq M} \mathfrak{a}_{L} = \mathfrak{a} - \bigcup_{\alpha \in \Sigma} \ker(\alpha) \,, \]
and the generic set
\[ \mathfrak{a}^{\operatorname{gen}} = \mathfrak{a}^{\operatorname{reg}} - \bigcup_{L \neq G} \mathfrak{a}^{L} \,. \]
Combined superscripts correspond to intersections: $\mathfrak{a}^{\operatorname{gen}, +} = \mathfrak{a}^{\operatorname{gen}} \cap \mathfrak{a}^{+}$.

We also define $(\mathfrak a^*)^{\operatorname{reg}}$, $(\mathfrak a^*)^{\operatorname{gen}}$ and $(\mathfrak a^*)^{+}$ to be the corresponding subsets under the isomorphism  $\mathfrak a \cong \mathfrak a^*$ defined by the Killing form. When $H \in \mathfrak a$ corresponds to $\lambda \in \mathfrak a^*$ under the isomorphism given by the Killing form, then $H \in \mathfrak a^{\operatorname{reg}}$ if and only if $\lambda$ is not orthogonal to any roots, and $H  \in \mathfrak a^{\operatorname{gen}}$ if and only if $\lambda$ is in addition not contained in a proper subspace spanned by roots.

We will frequently use the following lemma, so we take care to properly reference it.
We will frequently use the following lemmas, so we take care to properly reference them.

\begin{lemma} \label{conjugatetoAimplieslevi}Let $g \in G$ and $H \in \mathfrak{a}$. If $\operatorname{Ad}_{g}(H) \in \mathfrak{a}$, then $g \in M' Z_{G}(H)$.\end{lemma}

\begin{proof}This is stated in \cite[\S 5, Lemma~1]{harishchandra1975}. When $g \in K$, the statement gives precisely the degree of uniqueness in the $KAK$ decomposition of $G$, and a proof can be found in \cite[Lemma 7.38]{Knapp2002}. The general case can be reduced to $g \in K$ as follows. Write $g = kp$ in the Cartan decomposition. Then $\operatorname{Ad}_{p}(H) \in \operatorname{Ad}_{k^{-1}}(\mathfrak{a}) \subset \mathfrak{p}$, and \cite[\S V.24.C, Proposition 1]{borel1991} implies that $p \in Z_{G}(H)$. Then $\operatorname{Ad}_{k}(H) \in \mathfrak{a}$, and the conclusion follows from the $g \in K$ case.
\end{proof}

\begin{lemma}\label{truelemma}Let $g \in G$ and $H \in \mathfrak{a}$. If $\operatorname{Ad}_{g}(H) \in \mathfrak m \oplus \mathfrak{a}$, then $g \in M' Z_{G}(H)$.\end{lemma}

\begin{proof}If we can show that $\operatorname{Ad}_g(H) \in \mathfrak a$, the claim follows from Lemma~\ref{conjugatetoAimplieslevi}.

Consider the adjoint embedding $\operatorname{ad} : \mathfrak{g} \to \mathfrak{sl}(\mathfrak{g})$. Equip $\mathfrak g$ with any orthonormal basis compatible with the restricted root space decomposition \eqref{rootspacedecomp}. In such a basis, $\operatorname{ad}$ sends elements of $\mathfrak{a}$ to diagonal matrices and elements of $\mathfrak{k}$ to antisymmetric matrices.

Write $\operatorname{Ad}_g(H) = X + H'$ with $X \in \mathfrak{m}$ and $H' \in \mathfrak a$. In the chosen basis, $\operatorname{ad}_{ H'}$ is diagonal with real eigenvalues, and the antisymmetric matrix $\operatorname{ad}_{X}$ is diagonalizable with purely imaginary eigenvalues. Because $[H, X] = 0$, the elements $\operatorname{ad}_{ H'}$ and $\operatorname{ad}_{X}$ are simultaneously diagonalizable, so that the eigenvalues of $\operatorname{ad}_{ H'}+\operatorname{ad}_{X}$ are those of $\operatorname{ad}_{ H'}$ plus those of $\operatorname{ad}_{X}$, in a suitable ordering. If these eigenvalues are real, it must be that $X = 0$.

This proves that $\operatorname{Ad}_g(H) \in \mathfrak a$, and the lemma follows.
\end{proof}

\begin{lemma}We have $Z_G(A) = MA$ and $N_G(A) = M' A$.\end{lemma}

\begin{proof}
The first statement follows from \cite[Proposition 7.25]{Knapp2002}; the second statement follows by combining it with Lemma~\ref{conjugatetoAimplieslevi}.
\end{proof}

\subsection{Derivatives}

\label{sectionderivatives}

When $G$ is any Lie group with Lie algebra $\mathfrak g$ and $b$ is an element of the universal enveloping algebra $U(\mathfrak{g})$, we denote by $L_{b}$ the corresponding left invariant differential operator on $C^{\infty}(G)$
. When $X \in \mathfrak g \subset U(\mathfrak g)$, by definition
\[ (L_Xf)(g) = \left.\frac d{dt}\right\rvert_{t=0} f( ge^{tX}) \,.\]
When $f : M \to N$ is a differentiable map between differentiable manifolds, denote its differential at $m \in M$ by $(Df)_m$.
Using left translation we identify all tangent spaces $T_{g} G$ with $\mathfrak{g}$. When $g \in G$, denote by $L_{g}$ and $R_{g}$ the left and right multiplication by $g$ on $G$. With our convention on tangent spaces, we then have for all $g, h \in G$ that
\begin{align}
(D L_{g})_{h} & = \operatorname{id} \,,  \nonumber \\
(D R_{g})_{h} & = \operatorname{Ad}_{g^{-1}} \,. 
\end{align}
When $X, Y \in \mathfrak{g}$ we have
\begin{align}
L_{X} \operatorname{Ad}_{g}(Y) & = \operatorname{Ad}_{g}([X, Y]) \,, \label{derivativeAd} \\
L_{X} \operatorname{Ad}_{g^{-1}}(Y) & = - [X, \operatorname{Ad}_{g^{-1}}(Y)] \label{derivativeAdinverse} \,.
\end{align}

Assume now that $G$ is reductive as in the beginning of \S \ref{notationliegroups}.


\begin{lemma} \label{computationAfirstderivative} \label{differentialkpart}The differentials of $n$, $a$ and $\kappa$ at $g \in G$ are as follows:
\begin{align*}
(D n)_{g} & = \operatorname{Ad}_{a(g)} \circ E_{\mathfrak{n}} \circ \operatorname{Ad}_{\kappa(g)} \,, \\
(D a)_{g} & = E_{\mathfrak{a}} \circ \operatorname{Ad}_{\kappa(g)} \,, \\
(D \kappa)_{g} & = \operatorname{Ad}_{\kappa(g)^{-1}} \circ E_{\mathfrak{k}} \circ \operatorname{Ad}_{\kappa(g)} \,.
\end{align*}
\end{lemma}

\begin{proof}Write $g = nak$ and take $X \in \mathfrak{g}$. For the $N$-projection, write
\begin{align*}
n(ge^{X}) & = n \cdot n(ake^{X}) \\
& = n \cdot n(a k e^{X}k^{-1}) \\
& = n \cdot (a \cdot n(e^{\operatorname{Ad}_{k}(X)}) \cdot a^{-1}) \,.
\end{align*}
In the last step, we have used that $n(ah) = a n(h) a^{-1}$ for $a \in A$ and $h \in G$. Therefore 
\[ (Dn)_g = (D L_n)_e \circ (D (\left.\operatorname{Ad}_a\right\rvert_N))_e \circ (Dn)_e \circ \operatorname{Ad}_k \,. \] The first statement now follows from the fact that $(Dn)_{e} = E_{\mathfrak{n}}$. The other statements are proved similarly.
For the $A$-projection we write
\[
a(ge^{X}) = a \cdot a(ke^{X}) = a \cdot a(e^{\operatorname{Ad}_{k}(X)}) 
\]
and use that $(D a)_{e} = E_{\mathfrak{a}}$.
For the $K$-projection we write
\[
\kappa(ge^{X})  = \kappa(k e^{X}) \\
= \kappa(e^{\operatorname{Ad}_{k}(X)}) \cdot k
\]
and use that $(D \kappa)_{e} = E_{\mathfrak{k}}$.
\end{proof}

The differential of the $A$-projection is also computed in \cite[Lemma~6.1]{marshall2015}. Even though the proof there is for $G = \operatorname{SL}_3(\mathbb R)$, the argument works in general. Compare also \cite[Corollary~5.2]{duistermaat1983}, but note that the Iwasawa decomposition used there is $KAN$ instead of $NAK$. 

\begin{lemma} \label{computationAsecondderivative} For $X, Y \in \mathfrak{g}$ and $g \in G$ we have
\begin{align*}
(L_{X}L_{Y} a)_g & = E_{\mathfrak{a}}( [E_{\mathfrak{k}}(\operatorname{Ad}_{\kappa(g)}(X)), \operatorname{Ad}_{\kappa(g)}(Y)] ) \\
& = E_{\mathfrak{a}}( [ E_{\mathfrak{k}}(\operatorname{Ad}_{\kappa(g)}(X)), E_{\mathfrak{n}}( \operatorname{Ad}_{\kappa(g)}(Y) ) ] ) \\
& = E_{\mathfrak{a}}( [ \operatorname{Ad}_{\kappa(g)}(X), E_{\mathfrak{n}}( \operatorname{Ad}_{\kappa(g)}(Y) ) ] ) \\
&= \sum_{\alpha \in \Sigma^+} \langle\theta R_{-\alpha}(\operatorname{Ad}_{\kappa(g)}(X)), (R_{\alpha} - \theta R_{-\alpha}) ( \operatorname{Ad}_{\kappa(g)}(Y) ) \rangle_\theta \cdot H_\alpha  \,.
\end{align*}
\end{lemma}

\begin{proof}
Similar to the proof of \cite[Lemma~6.1]{duistermaat1983}. Alternatively, by Lemma~\ref{computationAfirstderivative} we find
\[ L_{Y} a(g) = E_{\mathfrak{a}} (\operatorname{Ad}_{\kappa(g)}(Y) ) \,. \]
Using the chain rule, \eqref{derivativeAd} and Lemma~\ref{differentialkpart} to compute $(D \kappa)_{g}$ we have
\begin{align*}
 L_{X} L_{Y} a(g) & = E_{\mathfrak{a}} (\operatorname{Ad}_{\kappa(g)}([(D \kappa)_{g} (X), Y]) ) \\
& = E_{\mathfrak{a}}( [E_{\mathfrak{k}}(\operatorname{Ad}_{\kappa(g)}(X)), \operatorname{Ad}_{\kappa(g)}(Y)] ) \,.
\end{align*}
This is the first equality. The other equalities follow as in \cite[Lemma~6.1]{duistermaat1983}.
\end{proof}

\section{The \texorpdfstring{$N$}{N}-projection and the Gram--Schmidt process}
\label{secNprojection}

Unless otherwise specified, $G$ is a semisimple Lie group as in \S\ref{notationliegroups}. In this subsection we prove Theorem~\ref{npartbounded} and the stronger Theorem~\ref{npartboundeduniform}.


Recall that the Iwasawa decomposition for $\operatorname{SL}_n(\mathbb R)$ is nothing but the Gram--Schmidt process: Take $g \in \operatorname{SL}_n(\mathbb R)$. There is a unique $n \in N$ such that the rows of $n^{-1}g$ are orthogonal for the Euclidean inner product on $\mathbb R^{n}$. There is a unique $a \in A$ such that the rows of $k := a^{-1}n^{-1}g$ have norm $1$. The Iwasawa decomposition of $g$ is then $nak$.

\subsection{Intro: \texorpdfstring{$\operatorname{SL}_2(\mathbb R)$}{SL2(R)}}

\label{npartsltwo}

We first prove Theorem~\ref{npartbounded} when $G = \operatorname{SL}_2(\mathbb R)$. This is quite trivial but already gives a good idea of what is happening.

\begin{proof}[Proof of Theorem~\ref{npartbounded} when $G = \operatorname{SL}_2(\mathbb R)$] Write $g = \begin{psmallmatrix}
v \\w
\end{psmallmatrix}$, and view $v, w \in \mathbb R^2$ as row vectors. Let $y > 0$. Multiplication on the right by $a = \operatorname{diag}(y, y^{-1}) \in A$ corresponds to letting the matrix $a$ act on $v$ and $w$. The projection $n(ga)$ is the matrix $\begin{psmallmatrix}
1 & x \\
0 & 1
\end{psmallmatrix}$ where
\[ x = \frac{\langle va, wa \rangle}{\left\langle wa, wa \right\rangle} \,. \]
 It is the unique matrix $n \in N$ for which the rows of $n^{-1}ga$ are orthogonal. We must show that $x$ is bounded. In the generic case where both coordinates of $w$ in the standard basis are nonzero, the denominator $\left\langle wa, wa \right\rangle$ is $\asymp \max(y^2, y^{-2})$, because $a$ acts by $y$ in one coordinate and by $y^{-1}$ in the other. The numerator is always $\ll \max(y^2, y^{-2})$, for the same reason. Therefore in this generic case, $x$ is bounded and the claim follows. When a coordinate of $w$ vanishes, the corresponding term in the denominator $\left\langle wa, wa \right\rangle$ is nonzero, but the same is then true in the numerator. So the same bounds hold, with $\max(y^2, y^{-2})$ replaced by $y^{2}$ or $y^{-2}$, depending on which coordinate of $w$ vanishes. The conclusion follows.
\end{proof}

\subsection{\texorpdfstring{$\operatorname{SL}_n(\mathbb R)$}{SLn(R)}, exterior powers}

\label{npartsln}

For $\operatorname{SL}_n(\mathbb R)$ the same phenomenon occurs, where if a term in a denominator of a certain fraction vanishes, then so does the corresponding term in the numerator. The denominators will here be norms of orthogonal projections onto the orthocomplement of previous vectors, and they are most conveniently expressed using norms on exterior powers.

When $(V, b)$ is a bilinear space of finite dimension over a field, $b$ can be identified with a linear map $V \to V^*$. It induces a linear map $\bigwedge^k V \to \bigwedge^k V^* \cong (\bigwedge^k V)^*$ for all $k \geq 0$, where the identification $\bigwedge^k V^* \cong (\bigwedge^k V)^*$ comes from the natural pairing between $\bigwedge^k V$ and $\bigwedge^k V^*$. We get a natural bilinear form $\bigwedge^k b$ on $\bigwedge^k V$. Assume now that the field is $\mathbb R$, and that $b$ is symmetric positive definite. Then so is $\bigwedge^k b$. We denote all induced bilinear forms by $\langle \cdot, \cdot \rangle$ for convenience. If $(e_i)$ is an orthonormal basis of $V$, then an induced basis of $\bigwedge^k V$ consisting of elements of the form $e_{j_1} \wedge \cdots \wedge e_{j_k}$ is also orthonormal. As a consequence, if $v, w_1, \ldots, w_k \in V$ with the $w_i$ linearly independent, $W = \operatorname{span}(w_i)$ and $\operatorname{pr}_{W^\perp}$ denotes the orthogonal projection onto $W^\perp$, then
\begin{equation}\label{normprojwedgesingle}
\lVert \operatorname{pr}_{W^\perp}(v) \rVert = \frac{\left\lVert v \wedge \left(\bigwedge_{i=1}^k w_i\right)\right\rVert}{\left\lVert \bigwedge_{i=1}^k w_i \right\rVert} \,,
\end{equation}
where the norms are those induced on $V$, $\bigwedge^{k+1} V$ and $\bigwedge^{k} V$. More generally, for $v_1, v_2 \in V$ we have
\begin{equation}\label{orthprojectionnormintermsofwedge}
\left\langle \operatorname{pr}_{W^\perp}(v_1), \operatorname{pr}_{W^\perp}(v_2) \right\rangle = \frac{\left\langle v_1 \wedge \left(\bigwedge_{i=1}^k w_i\right), v_2 \wedge \left(\bigwedge_{i=1}^k w_i\right)  \right\rangle}{\left\lVert \bigwedge_{i=1}^k w_i \right\rVert^2} \,.
\end{equation}
We also have the $i$th coefficient of $\operatorname{pr}_W(v)$ in the basis $(w_i)$ is given by
\begin{equation}\label{orthoprojectioncoeffintermsofwedge}
\frac{\left \langle w_1 \wedge \cdots \wedge \widehat{w_i} \wedge v\wedge \cdots \wedge w_{k}, \bigwedge_{i=1}^k w_i \right \rangle}{\left\lVert \bigwedge_{i=1}^k w_i \right\rVert^2} \,,
\end{equation}
where the hat denotes omission. 

\begin{proof}[Proof of Theorem~\ref{npartbounded} when $G = \operatorname{SL}_n(\mathbb R)$] Write $g$ as a column of row vectors: $g = (v_1, \ldots, v_n)^T$ with the $v_i \in \mathbb R^n$, and view $a \in A$ as acting on $\mathbb R^n$ diagonally. Define $V_i = \operatorname{span}(v_{i}, \ldots, v_n)$, so that $V_{n+1} = \{0\}$. Define $w_i^a$ to be the $i$th row of $n(ga)^{-1}ga$. Equip $\mathbb R^n$ with the Euclidean inner product. The first step in the Gram--Schmidt process applied to $ga$ is the recurrence relation
\begin{align*}
w_{n}^a &= v_n a\,, \\
w_{i-1}^a &= \operatorname{pr}_{(V_{i}a)^\perp} (v_{i-1}a) \\
&= v_{i-1}a - \sum_{j=i}^n \frac{\langle v_{i-1}a, w_j^a \rangle}{\langle w_j^a, w_j^a \rangle }w_j^a \qquad (i = n, \ldots, 2) \,.
\end{align*}
In the above sum, the coefficient in the term with index $j$ is the $(i-1,j)$th entry of $n(ga)$. We must show that those coefficients are bounded independently of $a$. For such a coefficient, with $i \leq j$, we have
\begin{align*}
\frac{\langle v_{i-1}a, w_j^a \rangle}{\langle w_j^a, w_j^a \rangle }
&= \frac{\left\langle \operatorname{pr}_{(V_{j+1}a)^\perp}(v_{i-1}a), \operatorname{pr}_{(V_{j+1}a)^\perp}(v_ja)\right\rangle}{\left\lVert \operatorname{pr}_{(V_{j+1}a)^\perp}(v_ja)\right\rVert^2}  \\
&= \frac{\left\langle v_{i-1}a \wedge \left(\bigwedge_{k = j+1}^n v_k a\right), \bigwedge_{k = j}^n v_k a \right\rangle}{\left\lVert \bigwedge_{k = j}^n v_k a \right\rVert^2} \,,
\end{align*}
where we have used \eqref{orthprojectionnormintermsofwedge}.
The action of $A$ on $\bigwedge^{n-j+1} \mathbb R^n$ is still diagonal in an induced basis. Fix such a basis and write $a = \operatorname{diag}(a_1, \ldots, a_K)$ in that basis. If $c_1, \ldots, c_K$ denote the coordinates of $v_{i-1} \wedge \left(\bigwedge_{k = j+1}^n v_k \right)$ and $d_1, \ldots, d_K$ those of $\bigwedge_{k = j}^n v_k$, then by the triangle inequality
\begin{align}\label{Nprojtriangleinequality}
\begin{split}
\left\lvert\frac{\langle v_{i-1}a, w_j^a \rangle}{\langle w_j^a, w_j^a \rangle }\right\rvert &= \left\lvert\frac{\sum_{k=1}^K a_k^2 c_k d_k}{\sum_{k=1}^K a_k^2 d_k^2} \right\rvert \\
&\leq \max_{d_k \neq 0} \left\lvert \frac{c_kd_k}{d_k^2}\right\rvert \,.
\end{split}
\end{align}
This bound does not depend on $a$, so that the entries of $n(ga)$ are bounded.
\end{proof}

\subsection{Semisimple groups}

\label{Npartsemisimplegroup}

Now let $G$ be a semisimple Lie group. There is no real reason to assume that $G$ has finite center, and Theorem~\ref{npartbounded} is insensitive to central extensions in any case, but we do so because our setup in \S\ref{notationliegroups} is under this assumption. Let $\rho : G \to \operatorname{GL}(V)$ be any finite-dimensional representation with discrete kernel, such as the adjoint representation $\operatorname{Ad} : G \to \operatorname{GL}(\mathfrak g)$, or the standard representation $\operatorname{Std}$ if $G$ is already linear. The fact that $\rho$ can be arbitrary is not essential, but it leads to questions about uniformity.

\begin{remark}Necessarily $\rho(G) \subset \operatorname{SL}(V)$, because $\mathfrak g = [\mathfrak g,  \mathfrak g]$ consists of commutators. Note that $\rho(G)$ is automatically closed, by \cite[Proposition 7.9]{Knapp2002} or alternatively by properties of Malcev closure \cite[\S1.4.2, Theorem 3]{onishchik1980}.\end{remark}

\begin{lemma}\label{compatibleiwasawa}In a suitable basis of $V$, the groups $\rho(N)$, $\rho(A)$ and $\rho(K)$ are contained in the standard Iwasawa components of $\operatorname{GL}(V)$.\end{lemma}

\begin{proof}As in the proof of \cite[Proposition 7.9]{Knapp2002}, there is a basis of $V$ such that $\rho(K) \subset \operatorname{SO}_n(\mathbb R)$ and $\rho(P)$ consists of symmetric matrices. Because $A \subset P$ is commutative, by acting on the basis by a suitable element of $\operatorname{SO}_n(\mathbb R)$ we may assume that $\rho(A)$ is diagonal. Then the basis consists of restricted weight vectors. Sorting them by non-increasing weight ensures that $\rho(N)$ is upper triangular unipotent.
\end{proof}

Equip $V$ with an Iwasawa-compatible basis given by Lemma~\ref{compatibleiwasawa} and denote the corresponding Iwasawa decomposition of $\operatorname{SL}(V)$ by $N' A' K'$. Denote the $N'$-projection on $\operatorname{SL}(V)$ by $n'$. 

\begin{proof}[Proof of Theorem~\ref{npartbounded}] By the case $G = \operatorname{SL}_n(\mathbb R)$ proved in \S\ref{npartsln}, the projection $n'(\rho(gA)) \subset n'(\rho(g) A')$ is relatively compact in $N'$. It is contained in the closed set $\rho(N)$, and therefore relatively compact in $\rho(N)$. Because $\rho$ has central kernel contained in $K$, it induces an isomorphism $N \to \rho(N)$. Therefore $n(gA) \cong n'(\rho(gA))$ is relatively compact in $N$.
\end{proof}

\subsection{Uniformity}

\label{npartpartitionnoncanonical}

For $G = \operatorname{SL}_2(\mathbb R)$ it is apparent from the proof in \S\ref{npartsltwo} that uniformity holds in Theorem~\ref{npartbounded} when the entries of $g$ are bounded and both entries on the second row of $g \in G$ are bounded away from $0$. We generalize this and partition any semisimple group $G$ into subsets for which uniformity holds on compact subsets. We briefly describe one such partition here and state a result about a more optimal partition in \S\ref{npartcanonicalpartitions}.

We use the notation from \S\ref{Npartsemisimplegroup}. Let $d = \dim(V)$ and let $\langle \cdot, \cdot \rangle$ be the Euclidean inner product for the chosen basis of $V$. For $\lambda \in \mathfrak a^*$ denote by $V_\lambda$ the restricted weight space. To prove Theorem~\ref{npartbounded}, instead of reducing to the case of $\operatorname{SL}_n(\mathbb R)$ as in \S\ref{Npartsemisimplegroup}, one can also directly use the argument in \S\ref{npartsln}, but restricted to elements $g \in \rho(G)$, and this leads to a statement with uniformity. For every tuple $\mathcal S = (\mathcal S_1, \ldots,\mathcal S_d)$  ($\mathcal S$ for support) of sets of integral weights of $\mathfrak g$, we may consider the subset $\Omega_{\mathcal S}$ of $G$ that consists of the elements $g$ with the following property: For every $j \in  \{1, \ldots, d \}$, the wedge product of every last $j$ rows of $\rho(g)$, as an element of $\bigwedge^j V$, has a nonzero component precisely along the weight spaces $(\bigwedge^j V)_\lambda$ with $\lambda \in \mathcal S_j$. It is then apparent that the upper bound analogous to \eqref{Nprojtriangleinequality} is uniform for $g$ in compact subsets of $\Omega_{\mathcal S}$, because the coefficients $d_k$ (which are now norms of weight space projections) are either zero or bounded away from zero.

\subsection{Canonical partitions}

\label{npartcanonicalpartitions}

The sets $\Omega_{\mathcal S}$ in \S\ref{npartpartitionnoncanonical} depend on the choice of basis of weight vectors of $V$. Moreover, they may form a partition of $G$ that is unnecessarily fine for the uniformity statement to be true; this can happen when the weight spaces of $V$ are not 1-dimensional. It is desirable to construct a basis-independent partition, which is what we do now, and which we expect to be the coarsest possible.

We continue to use the notation from \S\ref{Npartsemisimplegroup} and consider the right action of $\operatorname{GL}(V)$ on $V$ by transposition: $vg := g^T v$. That is, in the chosen basis of $V$, the rows of $g \in \operatorname{GL}(V)$ are the images of the basis elements under the right action of $g$. (This awkward definition is an artifact of our decision to write Iwasawa decompositions as $NAK$ rather than $KAN$.) Note that for $g \in G$ we have $\rho(g)^T = \rho(\theta(g))^{-1}$, because $\rho(A)$ consists of diagonal matrices and $\rho(K) \subset \operatorname{SO}_n(\mathbb R)$. That is, while the right action of $\operatorname{GL}(V)$ by transposition depends on the choice of basis, the restriction to $G$ depends only on the choice of Iwasawa decomposition of $G$ (and in fact, only on $A$ and $K$).

We will prove Theorem~\ref{npartboundeduniform} below in a way similar to the argument sketched in \S\ref{npartpartitionnoncanonical}, but using a block-by-block rather than a row-by-row orthogonalization process, where we transform a matrix so that a block of rows (corresponding to a weight space) is orthogonal to previous blocks of rows. That the Iwasawa decomposition corresponds to such a process, is the content of the following lemma.

For integral weights $\lambda, \mu$ of $\mathfrak g$, we write $\mu < \lambda$ if $\lambda - \mu$ is nonzero and is a nonnegative integer linear combination of positive roots.

\begin{lemma}\label{nprojectionintermsoforthogonalweight}Let $g \in G$ and $n' = n'(\rho(g))$. Then for distinct weights $\lambda, \mu$ of $V$ we have $V_\lambda n'^{-1} \rho(g)\perp V_\mu n'^{-1} \rho(g)$ and $V_\lambda (n'^{-1}-1) \perp V_\lambda$.\end{lemma}

\begin{proof}We have that $n'^{-1} \rho(g) \in A' K'$. The first statement follows from the facts that $V_\lambda A' = V_\lambda$, that $V_\lambda  \perp V_\mu$ and that $K'$ preserves orthogonality. For the second statement, we have for $X \in \operatorname{Lie}(N')$ that $X^T V_\lambda \in \bigoplus_{\mu < \lambda} V_\mu$, and exponentiating gives that $V_\lambda(N' -1)  \in \bigoplus_{\mu < \lambda} V_\mu$.
\end{proof}

Let $\Lambda$ be the sets of weights of $\mathfrak a$ in $V$, denote by $m_\lambda$ the multiplicity of $\lambda \in \Lambda$, and define $s_\lambda = \sum_{\mu < \lambda} m_\mu$. For any tuple $\mathcal S = (\mathcal S_\lambda)_{\lambda \in \Lambda}$ of sets of integral weights of $\mathfrak g$, consider the subset $\Omega_{\mathcal S}$ of $G$ that consists of the elements $g$ with the following property: For every $\lambda \in \Lambda$, the line $\bigwedge^{s_\lambda} \bigoplus_{\mu < \lambda} V_\mu \rho(g) \subset \bigwedge^{s_\lambda} V$ has nonzero orthogonal projection on the weight-$\alpha$ spaces $(\bigwedge^{s_\lambda} V)_\alpha$ for $\alpha \in S_\lambda$, and zero projection when $\alpha \notin \mathcal S_\lambda$.

\begin{theorem}\label{npartboundeduniform} Let $\mathcal S$ be any tuple as above, and $D \subset \Omega_{\mathcal S}$ compact. Then the $N$-projection $n(DA)$ is relatively compact.
\end{theorem}

Before proving Theorem~\ref{npartboundeduniform}, we state some basic properties of the sets $\Omega_{\mathcal S}$.

\begin{proposition}\label{coordOmegaspartitionG}The sets of the form $\Omega_{\mathcal S}$ (or those that are non-empty) partition $G$. They are stable on the left by $NAM$. \end{proposition}

\begin{proof}
That the $\Omega_{\mathcal S}$ partition $G$ is clear from their definition. Acting on the left by $\rho(NAM)$ on $g \in \operatorname{SL}(V)$ does not change the lines $\bigwedge^{s_\lambda} \bigoplus_{\mu < \lambda} V_\mu g$, so that the individual conditions defining $\Omega_{\mathcal S}$ are left-invariant under $NAM$.
\end{proof}

For $\lambda \in \Lambda$, define $\mathcal S^0_\lambda$ to be the set of weights $\alpha$ of $\mathfrak g$ such that the line
\[ \bigwedge^{s_\lambda} \bigoplus_{\mu < \lambda} V_\mu \rho(g) \subset \bigwedge^{s_\lambda} V \]
has nonzero projection onto the weight-$\alpha$ subspace for at least one $g \in G$. Define $\mathcal S^0 = (\mathcal S^0_\lambda)_{\lambda \in \Lambda}$. Then $\Omega_{\mathcal S^0}$ contains ``most'' elements of $G$.

\begin{proposition}\label{omegaopendense}The set $\Omega_{\mathcal S^0}$ is open and dense in $G$. \end{proposition}

\begin{proof}
 The image $\rho(G)$ is the identity component of the real points of a real algebraic closed subgroup of $\operatorname{SL}(V)$ \cite[\S3.3.3]{onishchik1980}. The set $\rho(\Omega_{\mathcal S^0})$ is defined by the non-vanishing of finitely many rational functions (corresponding to projections onto weight spaces) and therefore Zariski open in $\rho(G)$. (There are no vanishing conditions when $\mathcal S = \mathcal S^0$, or rather, they are satisfied on all of $\rho(G)$.) It follows that $\Omega_{\mathcal S^0}$ is open and dense in $G$.
\end{proof}

\begin{example}Take $G = \operatorname{SL}_2(\mathbb R)$ and $\rho = \operatorname{Std}$ the standard representation. Then $\Omega_{\mathcal S^0} = G - NAM'$, because the only defining condition is that both entries on the bottom row are nonzero. The only other non-empty sets of the form $\Omega_{\mathcal S}$ are $NAM$ and $NAM \cdot\begin{psmallmatrix}
0 & 1 \\
-1 & 0
\end{psmallmatrix}$. The geodesics in $G/K$ corresponding to elements of $\Omega_{\mathcal S^0}$ are semicircles. The other $\Omega_{\mathcal S}$ give rise to the vertical geodesics, with both orientations. One can check using an explicit computation that $\rho = \operatorname{Ad}$ gives the same partition of $\operatorname{SL}_2(\mathbb R)$. 
\end{example}

\begin{example}\label{omegaslthree}Take $G = \operatorname{SL}_3(\mathbb R)$ and $\rho = \operatorname{Std}$. The sets $\Omega_{\mathcal S}$ are determined by their $K$-projection, in view of Lemma~\ref{coordOmegaspartitionG}. The sets $\kappa(\Omega_{S})$ come in all possible dimensions: The $0$-dimensional ones, which are the six right cosets of $M$ in $M'$. The $1$-dimensional ones: the three right $M'$-translates of $G \cap (\operatorname{O}(2) \times \operatorname{O}(1)) - M'$ and the three right $M'$-translates of $G \cap (\operatorname{O}(1) \times \operatorname{O}(2)) - M'$. The $2$-dimensional ones: the three right $M'$-translates of the product $(G  \cap (\operatorname{O}(2) \times \operatorname{O}(1)) - M' )(G \cap (\operatorname{O}(1) \times \operatorname{O}(2)) - M')$ and the three right $M'$-translates of that product with the two factors interchanged. And finally, the dense open set $\kappa(\Omega_{S^0})$.\end{example}

\begin{remark}For any $G$, the set $\kappa(\Omega_{S^0})$ is disjoint from the sets $LM'$ with $L \in \mathcal L$ a standard Levi subgroup of a semistandard parabolic; this follows already from the constraint corresponding to the minimal weights $\lambda$. But in general it is smaller than just the complement of the sets of the form $K \cap L M'$, as illustrated in Example~\ref{omegaslthree}.

It remains unclear to us whether the partition of $G$ into the $\Omega_{\mathcal S}$ depends on the choice of representation $\rho$, and if not, whether the partition is the coarsest possible (up to splitting into connected components) for uniformity to hold. It would also be desirable to have a concrete description of the partition in general, as we do have when $G = \operatorname{SL}_3(\mathbb R)$.\end{remark}

\begin{proof}[Proof of Theorem~\ref{npartboundeduniform}] Let $g \in G$ and $a \in A$. We view $n'(\rho(ga))^{-1}$ as a block matrix: For weights $\lambda \geq \mu$ of $V$, define
\begin{align*}
T_{\lambda, \mu} : V_\lambda &\to V_\mu \\
v & \mapsto \operatorname{pr}_{V_\mu} \left( v \cdot n'(\rho(ga))^{-1} \right)
\end{align*}
and define $T_\lambda = \sum_{\mu < \lambda} T_{\lambda, \mu} : V_\lambda \to \bigoplus_{\mu < \lambda} V_\mu$. We must show that each $T_\lambda$ is a bounded operator, uniformly in $a \in A$ and $g$ in compact subsets of each $\Omega_{\mathcal S}$. By Lemma~\ref{nprojectionintermsoforthogonalweight}, the map $T_\lambda$ satisfies
\[ ((1+ T_{\lambda}) V_\lambda) \rho(ga) \perp \bigoplus_{\mu < \lambda} V_\mu  \rho(ga) \,. \]
That is, $(T_{\lambda} v)  \rho(ga) = - \operatorname{pr}_{\bigoplus_{\mu < \lambda} V_\mu \rho(ga)} (v \rho(ga) )$. (These two identities express the block-by-block orthogonalization process.) To show that $T_\lambda$ is bounded, we may fix any basis $(b_i)_{1 \leq i \leq s_\lambda}$ of $\bigoplus_{\mu < \lambda} V_\mu $ and we must show that for every $v \in V_\lambda$ the coordinates of the orthogonal projection of $v \rho(ga)$ onto $\bigoplus_{\mu < \lambda} V_\mu \rho(ga)$ in the basis $(b_i \rho(ga))$ are bounded. By \eqref{orthoprojectioncoeffintermsofwedge}, the $i$th coordinate is equal to
\begin{align*}
\frac{\langle (b_1 \wedge \cdots \wedge \widehat{b_i} \wedge v\wedge \cdots \wedge b_{s_\lambda}) \rho(ga), (b_1 \wedge \cdots \wedge  b_{s_\lambda}) \rho(ga) \rangle}{\left\lVert (b_1 \wedge \cdots \wedge b_{s_\lambda})\rho(ga) \right\rVert^2} \,,
\end{align*}
where the hat denotes omission. We may bound this using an inequality similar to \eqref{Nprojtriangleinequality}. Concretely, call $d_\mu$ the projection of $(b_1 \wedge \cdots \wedge b_{s_\lambda})\rho(g)$ onto the weight-$\mu$ subspace of $\bigwedge ^{s_\lambda}V$, and $c_\mu$ that of $(b_1 \wedge \cdots \wedge \widehat{b_i} \wedge v\wedge \cdots \wedge b_{s_\lambda}) \rho(g)$. Then the fraction above equals
\[
\frac{\sum_{\mu} \mu(a)^2 \langle c_\mu, d_\mu \rangle}{\sum_{\mu} \mu(a)^2 \lVert d_\mu \rVert^2}\,, \]
where we define the weights on $A$ using the exponential map. By the triangle inequality, this is at most
\[ \max_{d_\mu \neq 0} \frac{\lvert\langle c_\mu, d_\mu \rangle \rvert}{\lVert d_\mu \rVert^2} \,,\]
which gives a uniform upper bound for $g$ in compact subsets of each $\Omega_{\mathcal S}$, because the $d_\mu$ are then either zero or are bounded away from zero.
\end{proof}

\section{The \texorpdfstring{$A$}{A}-projection and extreme points}
\label{secAprojection}

In this subsection we prove Theorem~\ref{aresultsallintro}. Unless otherwise stated, $G$ is a semisimple Lie group as in \S\ref{notationliegroups}. For $H_0 \in \mathfrak a$ and $g \in G$ define
\begin{align*}
h_{H_{0}, g} : A & \to \mathbb{R} \\
a & \mapsto \langle H_{0}, H(g a) \rangle \,.
\end{align*}
When $H_0$ corresponds to $\lambda \in \mathfrak a^*$ under the isomorphism given by the Killing form, this is precisely the function $h_{\lambda , g}$ in Theorem~\ref{aresultsallintro}. Recall from \S\ref{defgenericset} that $\lambda$ is nonsingular if and only if $H_0$ is, $\lambda$ lies in the positive chamber of $\mathfrak a^*$ if and only if $H_0 \in \mathfrak a^{+}$, and $\lambda$ does not lie in a proper subspace spanned by roots if and only if $H_0 \notin \bigcup_{L \in \mathcal L - \{ G\}} \mathfrak a^L$. Thus the condition on $\lambda$ in Theorem~\ref{aresultsallintro} is equivalent to $H_0 \in \mathfrak a^{\operatorname{gen}, +}$.

The first part of Theorem~\ref{aresultsallintro}, concerning uniqueness and nondegeneracy, is proved in \S\ref{sectionuniqueness}; see Corollary~\ref{criticalpointshnondegenerate} and Proposition~\ref{uniquenesscriticalpointcartan}. The second part of Theorem~\ref{aresultsallintro}, concerning existence, is proved in \S\ref{sectionstructurelevelsets} (although the hard work is done in \S\ref{sectionexistence}); see Corollary~\ref{levelsetGcodim} and Proposition~\ref{regularsetGopen}.

\subsection{Critical points and level sets}

Using Lemma~\ref{computationAfirstderivative} we find that the differential of $h_{H_{0}, g}$ at $a \in A$ is
\begin{align} \label{cartanorbitdifferential}
\begin{split}
(Dh_{H_{0}, g})_{a} : \mathfrak{a} & \to \mathbb{R} \\
H & \mapsto \langle H_{0}, \operatorname{Ad}_{\kappa(ga)}(H) \rangle \,.
\end{split}
\end{align}
When $G$ is any reductive group with all associated data as in \S \ref{notationliegroups}, and $H_{0} \in \mathfrak{a}$, define the set $\mathcal C(G, H_0)$ as follows:
\begin{equation}\label{definitionCreductive}
\mathcal{C}(G, H_{0}) = \{ k \in K : \operatorname{Ad}_{k^{-1}} (H_{0}) \perp \mathfrak{a} \} \,.
\end{equation}
We allow $G$ to be reductive because we will occasionally need to work with the sets $\mathcal C(L, H_0)$ for $L \in \mathcal L$.

Now let $G$ again be semisimple. The following two lemmas are clear from \eqref{cartanorbitdifferential} and the definition of $\mathcal{C}(G, H_{0})$.

\begin{lemma} \label{criticalpointsintermsofC}Let $H_{0} \in \mathfrak{a}$ and $g \in G$. Then $a \in A$ is a critical point of $h_{H_{0}, g}$ if and only if $\kappa(ga) \in \mathcal{C}(G, H_{0})$. \hfill $\square$\end{lemma}

\begin{lemma}\label{levelsetszerocritpoint}The set of $k \in K$ for which $1$ is a critical point of $h_{H_0, k}$ equals $\mathcal C(G, H_0)$. More generally, the set of $k \in K$ for which $a \in A$ is a critical point of $h_{H_0, k}$ equals $\kappa(\mathcal C(G, H_0)a^{-1})$.\hfill $\square$\end{lemma}

In view of Lemma~\ref{levelsetszerocritpoint}, we refer to the sets $\kappa(\mathcal C(G, H_0)a^{-1})$ as level sets. This will be fully justified after we show uniqueness of critical points in Proposition~\ref{uniquenesscriticalpointcartan}.

By decomposing $g$ in the Iwasawa decomposition we see that
\begin{equation}\label{heightfunctionreducetokpart}
h_{H_{0}, g}(a) = \langle H_{0}, H(g) \rangle + h_{H_{0}, \kappa(g)}(a) \,. 
\end{equation}
The first term in the right hand side being a mere constant, we see that the critical points of $h_{H_0, g}$ coincide with those of $h_{H_0, \kappa(g)}$. It therefore suffices to understand the critical points of $h_{H_{0}, k}$ for $k \in K$.

\begin{example} When $G = \operatorname{PSL}_2(\mathbb R)$ it is clear from the geometric picture in the introduction that a critical point of $h_{H_{0}, g}$ is a point $a \in A$ such that the tangent line to $gAi$ at $gai$ is horizontal. This is the content of the above Lemma~\ref{criticalpointsintermsofC} in this case. The set $\mathcal{C}(G, H_{0})$ consists of the two elements
\begin{align*}
c_{\pm} = \begin{pmatrix}
\cos(\pi /4) & \pm \sin(\pi /4) \\
\mp \sin(\pi /4) & \cos(\pi /4)
\end{pmatrix}
\end{align*}
corresponding to the two horizontal directions. See Figure~\ref{sltwocpicture} for a picture.
\end{example}

\begin{figure}[h!]
\begin{center}
\captionsetup{justification=centering}
\includegraphics[scale=1,angle=0]{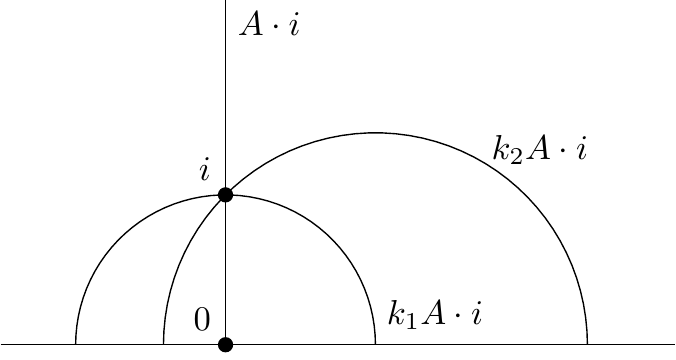}
\caption{Geodesics corresponding to an element $k_1 \in \mathcal C(G, H_0)$, and to an element $k_2$ that does not lie in $\mathcal C(G, H_0)$.}
\label{sltwocpicture}
\end{center}
\end{figure}

\begin{example}\label{examplemidpointdepends} Let $G = \operatorname{SL}_3(\mathbb R)$. Take elements $H_0, H_0' \in \mathfrak a$ that are not proportional, and take $k\in K$. We claim that $a \in A$ can never be a critical point of both $h_{H_0, k}$ and $h_{H_0', k}$. Indeed, replacing $k$ by $\kappa(ka)$ we may assume that $a = 0$, and that $k \in \mathcal C(G, H_0) \cap \mathcal C(G, H_0')$. That is, the symmetric matrices $\operatorname{Ad}_{k^{-1}}(H_0)$ and $\operatorname{Ad}_{k^{-1}}(H_0')$ have zeros on the diagonal. Because $\mathfrak a$ is $2$-dimensional the same is then true for any $H_0'' \in \mathfrak a$. Take now $H_0'' = H_0^2$. The matrix $\operatorname{Ad}_{k^{-1}}(H_0'') = \operatorname{Ad}_{k^{-1}}(H_0)\operatorname{Ad}_{k^{-1}}(H_0)^T$ has zeros on the diagonal on the one hand, but its trace is the sum of squares of the entries of $H_0$ on the other hand. This is a contradiction.\end{example}

\subsection{Existence of critical points}

\label{sectionexistence}

We want to show that there exists $k \in K$ with the property that $h_{H_0, k}$ has a critical point. Equivalently, in view of Lemma~\ref{levelsetszerocritpoint}, that the set $\mathcal C(G, H_0)$ is nonempty. Before proving that this is indeed the case for generic $H_0 \in\mathfrak a$, we need a negative result.

When $k \in K$ centralizes a nonzero subspace $V \subset \mathfrak a$, then $a \mapsto H(ka)$ grows linearly in the directions of $V$. In particular, for $h_{H_0, k}$ to have a critical point, $H_0$ must be orthogonal to $V$. Those critical points behave badly, and this is the reason to impose that $H_0 \notin \bigcup_{L \in \mathcal L - \{ G\}} \mathfrak a^{L}$ in Theorem~\ref{aresultsallintro}. We make this more precise in Lemma~\ref{mLnocriticalpoints}.

\begin{lemma} \label{kpartofmL}When $L \in \mathcal{L}$ and $m \in M'$, we have $\kappa(mL) \subset mL$.\end{lemma}

\begin{proof}
It follows from \cite[Proposition 7.25, Proposition 7.31]{Knapp2002} that when $L' \subset G$ is a semistandard Levi subgroup, one has $\kappa(L') \subset L'$. We may apply this to $L' = m L m^{-1}$.
\end{proof}

\begin{proposition} \label{mLnocriticalpoints}When $H_{0} \notin \bigcup_{L \in \mathcal L - \{ G\}}\mathfrak{a}^{L}$ and $k \in \bigcup_{L \in \mathcal L - \{ G\}} M' L$, the function $h_{H_{0}, k}$ has no critical point.
\end{proposition}

\begin{proof}Assume $a \in A$ is a critical point of $h_{H_{0}, k}$, and let $m \in M'$ and $L \in \mathcal L$ be such that $k \in mL$. By Lemma~\ref{kpartofmL} we also have $\kappa(ka) \in mL$. From \eqref{cartanorbitdifferential}, letting $H$ vary in $\mathfrak a_L$, we see that $H_0 \perp \operatorname{Ad}_m(\mathfrak a_L) = \mathfrak a_{\operatorname{Ad}_m(L)}$. That is, $H_0 \in \mathfrak a^{\operatorname{Ad}_m(L)}$, which contradicts our assumption.
\end{proof}

An equivalent formulation of Proposition~\ref{mLnocriticalpoints} is the following.

\begin{lemma}\label{CdoesnotmeetML} When $H_{0} \notin \bigcup_{L \in \mathcal L - \{ G\} }\mathfrak{a}^{L}$, the set $\mathcal{C}(G, H_{0})$ does not meet the set $\bigcup_{L \in \mathcal L - \{ G\}} M' L$.\end{lemma}

\begin{proof}
The elements $k \in \mathcal{C}(G, H_{0})$ have the property that $h_{H_0, k}$ has a critical point, namely $1$.
\end{proof}

When $G$ is reductive, the set $\mathcal C(G, H_0)$ is trivially empty when $H_{0}$ has a component along $Z(\mathfrak{g})$. That is, when $H_0 \notin \mathfrak g_{ss}$. In particular, when $L$ is a semistandard Levi subgroup of $G$, the set $\mathcal{C}(L, H_{0})$ is empty if $H_{0} \notin (\mathfrak{a} \cap \mathfrak{l}_{ss}) = \mathfrak{a}^{L}$. To study the sets $\mathcal{C}(G, H_{0})$ with $G$ reductive and $H_{0} \in \mathfrak{a} \cap \mathfrak{g}_{ss}$, consider the map
\begin{align}\label{definitionfmapC}
\begin{split}
f_{G, H_{0}} : K & \to \mathfrak{a} \cap \mathfrak{g}_{ss} \\
k & \mapsto E_{\mathfrak{a}} (\operatorname{Ad}_{k^{-1}}(H_{0})) \,.
\end{split}
\end{align}
By definition, $\mathcal{C}(G, H_{0}) = f_{G, H_{0}}^{-1}(0)$. By \eqref{derivativeAdinverse}, the differential of $f_{G, H_{0}}$ at $k \in K$ is given by
\begin{align} \label{differentialf}
\begin{split}
(Df_{G, H_{0}})_{k} : \mathfrak{k} & \to \mathfrak{a} \cap \mathfrak{g}_{ss} \\
X & \mapsto - E_{\mathfrak{a}} ([X, \operatorname{Ad}_{k^{-1}}(H_{0})] ) \,.
\end{split}
\end{align}
For $G$ semisimple, we want to prove that the sets $\mathcal{C}(G, H_{0})$ are generically nonempty. To this end, define
\begin{align*}
g_{G, H_{0}} : K & \to \mathbb{R}_{\geq 0} \\
k & \mapsto \lVert f_{G, H_{0}}(k) \rVert^{2} \,.
\end{align*}
We have $\mathcal{C}(G, H_{0}) = g_{G, H_{0}}^{-1}(0)$. Because $g_{G, H_{0}}$ is a continuous function on a compact set, it reaches a minimum. Our aim is to show that the minima of $g_{G, H_{0}}$ satisfy $g_{G, H_{0}}(k) = 0$, it it will follow that $\mathcal{C}(G, H_{0})$ is nonempty.

Using Leibniz's rule and \eqref{differentialf}, the differential of $g_{G, H_0}$ at $k \in K$ is given by
\begin{align}
(Dg_{G, H_{0}})_{k}(X) & = 2 \langle (Df_{G, H_{0}})_{k}(X), f_{G, H_{0}}(k) \rangle \nonumber \\
& = -2 \langle [X, \operatorname{Ad}_{k^{-1}}(H_{0})], f_{G, H_{0}}(k) \rangle \label{differentialg} \,.
\end{align}

\begin{lemma} \label{criticalsetsubmersion}  Let $G$ be reductive and $H_{0} \in \mathfrak{a}^{\operatorname{reg}} \cap \mathfrak{g}_{ss}$. Then $f_{G, H_{0}}$ is a submersion at points $k \notin \bigcup_{L \in \mathcal L - \{ G\}} M' L$.
\end{lemma}

\begin{proof}Suppose $(Df_{G, H_{0}})_{k}$ is not surjective for some $k \in K$. Then there exists a nonzero $H \in \mathfrak{a} \cap \mathfrak{g}_{ss}$ such that $\langle [X, \operatorname{Ad}_{k^{-1}}(H_{0})] , H \rangle = 0$ for all $X \in \mathfrak{k}$. Using associativity of the Killing form, this is equivalent to
\[ [\operatorname{Ad}_{k^{-1}}(H_{0}), H] \perp \mathfrak{k} \,. \]
But $[\operatorname{Ad}_{k^{-1}}(H_{0}), H] \in [\mathfrak{p}, \mathfrak{p}] \subset \mathfrak{k}$, so that $[\operatorname{Ad}_{k^{-1}}(H_{0}), H] = 0$. That is, $\operatorname{Ad}_{k}(H) \in Z_{\mathfrak{p}}(H_{0})$. Because $H_{0}$ is regular, this implies $\operatorname{Ad}_{k}(H) \in \mathfrak{a}$ (see \cite[Lemma~6.50]{Knapp2002}), and Lemma~\ref{conjugatetoAimplieslevi} then implies that $k \in M' L$ with $L = Z_{G}(H)$. Because $H \notin Z(\mathfrak{g})$, this is a proper semistandard Levi subgroup. This proves the first statement.
\end{proof}

\begin{corollary} \label{maintermgenericsubmersionatcriticalsetreductive}  Let $G$ be reductive and $H_{0} \in \mathfrak{a}^{\operatorname{gen}} \cap \mathfrak{g}_{ss}$. Then $f_{G, H_0}$ is a submersion at the points of $\mathcal{C}(G, H_{0})$.\end{corollary}

\begin{proof}The set $\mathcal C(G, H_0)$ does not meet any of the sets $M' L$ by Lemma~\ref{CdoesnotmeetML}, and therefore $f_{G, H_0}$ is a submersion at points of $\mathcal C(G, H_0)$, by Lemma~\ref{criticalsetsubmersion}.
\end{proof}

\begin{lemma} \label{criticalpointgequations}Let $H_{0} \in \mathfrak{a}^{\operatorname{reg}}$. Define $\mathcal{D}_{H_{0}} = \{ k \in K : k \in Z_{G}(f_{G, H_{0}}(k)) \}$. The following hold:
\begin{enumerate}[label = (\roman*)]
\item The function $g_{G, H_{0}}$ is right invariant under $M'$.
\item The set of critical points of $g_{G, H_{0}}$ is $\mathcal{D}_{H_{0}} M'$.
\item Let $k \in \mathcal{D}_{H_{0}}$ and define $L = Z_{G}(f_{G, H_{0}}(k))$. Write $H_{0} = H_{L} + H^{L}$ with $H_{L} \in \mathfrak{a}_{L}$ and $H^{L} \in \mathfrak{a}^{L}$. Then $f_{G, H_{0}}(k) = H_{L}$ and $k \in \mathcal{C}(L, H^{L})$.
\end{enumerate}
\end{lemma}

\begin{proof}
\begin{enumerate}[label = (\roman*)]
\item We show that $f_{G, H_{0}}(km) = \operatorname{Ad}_{m^{-1}}(f_{G, H_{0}}(k))$ for $m \in M'$. Because the adjoint action of $M'$ on $\mathfrak{a}$ is isometric, it will then follow that $g_{G, H_{0}}$ is right invariant under $M'$.

Because $M'$ normalizes $\mathfrak{a}$, it normalizes the orthogonal complement $\mathfrak{a}^{\perp}$, so that the adjoint action of $M'$ commutes with $E_{\mathfrak{a}}$. This implies
\begin{align*}
f_{G, H_{0}}(km) & = E_{\mathfrak{a}}(\operatorname{Ad}_{m^{-1}k^{-1}}(H_{0})) \\
& = \operatorname{Ad}_{m^{-1}} (E_{\mathfrak{a}}( \operatorname{Ad}_{k^{-1}}(H_{0}))) \\
& = \operatorname{Ad}_{m^{-1}}(f_{G, H_{0}}(k)) \,.
\end{align*}
\item Let $k$ be a critical point of $g$. Using \eqref{differentialg} we have for all $X \in \mathfrak{k}$,
\begin{align*}
0 = (Dg_{G, H_{0}})_{k} & = -2 \langle E_{\mathfrak{a}}([X, \operatorname{Ad}_{k^{-1}}(H_{0})]), f_{G, H_{0}}(k) \rangle \,.
\end{align*}
That is,
\[ \langle \mathfrak{k}, [\operatorname{Ad}_{k^{-1}}(H_{0}), f_{G, H_{0}}(k)] \rangle = 0 \,. \]
But $[\operatorname{Ad}_{k^{-1}}(H_{0}), f_{G, H_{0}}(k)] \in [\mathfrak{p}, \mathfrak{p}] = \mathfrak{k}$, so that this must be zero. Because $H_{0} \in \mathfrak{a}^{\operatorname{reg}}$, this implies that $f_{G, H_{0}}(k) \in \operatorname{Ad}_{k^{-1}}(\mathfrak{a})$. By Lemma~\ref{conjugatetoAimplieslevi} it follows that $k \in M' Z_{G}(f_{G, H_{0}}(k))$. Replacing $k$ by an appropriate right translate under $M'$, this becomes $k \in Z_{G}(f_{G, H_{0}}(k))$. That is, $k \in \mathcal{D}_{H_{0}}$.
\item Take $k \in \mathcal{D}_{H_{0}}$. By definition of $L$ we have $f_{G, H_{0}}(k) \in \mathfrak{a}_{L}$. Writing $H_{0} = H_{L} + H^{L}$ and using that $k \in L$, we have $f_{G, H_{0}}(k) = E_{\mathfrak{a}}(\operatorname{Ad}_{k^{-1}}(H_{0})) = H_{L} + E_{\mathfrak{a}}(\operatorname{Ad}_{k^{-1}}(H^{L}))$. If we prove that $E_{\mathfrak{a}}(\operatorname{Ad}_{k^{-1}}(H^{L})) = 0$, the two statements follow. On the one hand $f_{G, H_{0}}(k) \in \mathfrak{a}_{L}$ implies $E_{\mathfrak{a}}(\operatorname{Ad}_{k^{-1}}(H^{L})) = f_{G, H_{0}}(k) - H_{L} \in \mathfrak{a}_{L}$. On the other hand, $H^{L} \perp \mathfrak{a}_{L}$ and $k \in L$ implies $\operatorname{Ad}_{k^{-1}}(H^{L}) \perp \mathfrak{a}_{L}$, and hence $E_{\mathfrak{a}}(\operatorname{Ad}_{k^{-1}}(H^{L})) \perp \mathfrak{a}_{L}$. It follows that $E_{\mathfrak{a}}(\operatorname{Ad}_{k^{-1}}(H^{L})) \in \mathfrak{a}_{L} \cap \mathfrak{a}^{L} = \{ 0 \}$. \qedhere
\end{enumerate}
\end{proof}

\begin{lemma} Let $H_{0} \in \mathfrak{a}^{\operatorname{reg}}$, $\mathcal{D}_{H_{0}}$ be as in Lemma~\ref{criticalpointgequations} and $k \in \mathcal{D}_{H_{0}}$ a critical point of $g_{G, H_{0}}$. Define $L$, $H_{L}$ and $H^{L}$ as in the same lemma. Then the Hessian of $g_{G, H_{0}}$ at $k$ satisfies
\begin{align}
\label{existenceHessiansimplified}
\begin{split}
\frac12 (\operatorname{Hess}_{k} g_{G, H_0}) (X, X)  & = - \left\lVert [X, H_{L}] \right\rVert^{2} - \langle [X, \operatorname{Ad}_{k^{-1}}(H^{L})], [X, H_{L}] \rangle \\
& \mathrel{\phantom{=}} + \left\lVert E_{\mathfrak{a}}([ X, \operatorname{Ad}_{k^{-1}}(H^{L})]) \right\rVert^{2}
\end{split}
\end{align}
for all $X \in \mathfrak{k}$.
\end{lemma}

\begin{proof}
Starting from \eqref{differentialg} and using Leibniz's rule and \eqref{derivativeAdinverse}, we have that the Hessian of $g$ at the critical point $k$ takes the form
\begin{align*}
\operatorname{Hess}_{k} g : \mathfrak{k} \times \mathfrak{k} & \to \mathbb{R} \\
(X, Y) & \mapsto 2 \left \langle E_{\mathfrak{a}}( [ Y, [ X, \operatorname{Ad}_{k^{-1}}(H_{0})]]) , f_{G, H_{0}}(k) \right\rangle \\
& \mathrel{\phantom{\mapsto}} + 2 \left \langle E_{\mathfrak{a}}( [ Y, \operatorname{Ad}_{k^{-1}}(H_{0})]) , E_{\mathfrak{a}}([ X, \operatorname{Ad}_{k^{-1}}(H_{0})]) \right\rangle \,.
\end{align*}
As $f_{G, H_{0}}(k) \in \mathfrak{a}$, we may drop the projection $E_{\mathfrak{a}}$ in the first term. Replacing $f_{G, H_{0}}(k)$ by $H_{L}$ and using associativity of the Killing form we obtain
\begin{align*}
& -2 \left \langle [ X, \operatorname{Ad}_{k^{-1}}(H_{0})] , [Y, H_{L}] \right\rangle \\
& \qquad + 2 \left \langle E_{\mathfrak{a}}( [ Y, \operatorname{Ad}_{k^{-1}}(H_{0})]), E_{\mathfrak{a}}([ X, \operatorname{Ad}_{k^{-1}}(H_{0})]) \right\rangle \,.
\end{align*}
The associated quadratic form on $\mathfrak{k}$ sends
\begin{align*}
X \mapsto -2 \left \langle [ X, \operatorname{Ad}_{k^{-1}}(H_{0})] , [X, H_{L}] \right \rangle + 2 \left\lVert E_{\mathfrak{a}}([ X, \operatorname{Ad}_{k^{-1}}(H_{0})]) \right\rVert^{2} \,.
\end{align*}
Writing $H_{0} = H_{L} + H^{L}$ and recalling that $k \in L$ centralizes $H_{L}$, this becomes
\begin{align*}
& - 2 \left\lVert [X, H_{L}] \right\rVert^{2} -2 \langle [X, \operatorname{Ad}_{k^{-1}}(H^{L})], [X, H_{L}] \rangle \\
& \quad + 2 \left\lVert E_{\mathfrak{a}}([ X, H_{L} + \operatorname{Ad}_{k^{-1}}(H^{L})]) \right\rVert^{2} \,.
\end{align*}
Because $H_{L} \in \mathfrak{a}$, associativity of the Killing form implies $E_{\mathfrak{a}}([X, H_{L}]) = 0$, so that the third term simplifies and we obtain \eqref{existenceHessiansimplified}.
\end{proof}

\begin{lemma} \label{criticalpointgminimum}Let $H_{0} \in \mathfrak{a}^{\operatorname{gen}}$ and $k \in K$ be a critical point of $g_{G, H_{0}}$ with positive semidefinite Hessian. Then $g_{G, H_{0}}(k) = 0$. \end{lemma}

\begin{proof}The proof is by contradiction. Assuming that $f_{G, H_{0}}(k) \neq 0$, we will construct a direction in which the Hessian of $g_{G, H_{0}}$ is negative definite at $k$.

Let $\mathcal{D}_{H_{0}}$ be as in Lemma~\ref{criticalpointgequations}. By Lemma~\ref{criticalpointgequations} and right invariance of $g$ under $M'$, we may assume that $k \in \mathcal{D}_{H_{0}}$. Let $L = Z_{G}(f_{G, H_{0}}(k))$ and let $H_{L}$ and $H^{L}$ be as in Lemma~\ref{criticalpointgequations}, so that $f_{G, H_{0}}(k) = H_{L}$ and $k \in \mathcal{C}(L, H^{L})$. 
Suppose that $H_L \neq 0$. We will construct $X \in \mathfrak{k}$ for which the third term in \eqref{existenceHessiansimplified} is zero, and for which the other terms are nonpositive and not both zero.

Suppose first that $H^{L} \notin \mathfrak{a}^{\operatorname{reg}}$. Then there exists a nonzero element $X \in \mathfrak{k}$ with $[X, \operatorname{Ad}_{k^{-1}}(H^{L})] = 0$. Indeed, if $\alpha \in \Sigma$ is such that $\alpha(H^{L}) = 0$, take a nonzero $X' \in (\mathfrak{g}_{\alpha} + \mathfrak{g}_{- \alpha}) \cap \mathfrak{k}$. Then $[X', H^{L}] = 0$, and we can take $X = \operatorname{Ad}_{k^{-1}}(X')$.

With this choice of $X$, the second and third terms in \eqref{existenceHessiansimplified} vanish. We have $[X, H_{L}] = [X, \operatorname{Ad}_{k^{-1}}(H_{0})] \neq 0$ by regularity of $H_{0}$, so that
\[ (\operatorname{Hess}_{k} g_{G, H_0}) (X, X) = -2 \cdot \left\lVert [X, H_{L}] \right\rVert^{2} < 0 \,. \]
This is a contradiction.

Suppose now that $H^{L} \in \mathfrak{a}^{\operatorname{reg}}$.
We first show that there exists $X \in \mathfrak{k}$ for which the second term in \eqref{existenceHessiansimplified} is strictly negative. After that, we will modify $X$ in such a way that the second term stays the same and such that the third becomes zero.

The second term in \eqref{existenceHessiansimplified} equals
\[ - \langle [\operatorname{Ad}_{k}(X), H^{L}], [\operatorname{Ad}_{k}(X), H_{L}] \rangle \,. \]
Write $\operatorname{Ad}_{k}(X)$ in the restricted root space decomposition \eqref{rootspacedecomp} as
\[ X = \sum_{\alpha \in \Sigma} X_{\alpha} \,. \]
Then the above equals
\begin{equation} \label{secondtermnonemptyproofroots}
- \sum_{\alpha \in \Sigma} \alpha(H^{L}) \alpha(H_{L}) \lVert X_{\alpha} \rVert_{\theta}^{2} \,.
\end{equation}
Because $H_{L} \neq 0$ by assumption, there exists $\beta \in \Sigma$ such that $\beta(H_{L}) \neq 0$. Because $H^{L} \in \mathfrak{a}^{\operatorname{reg}}$, we have $\beta(H^{L}) \beta(H_{L}) \neq 0$. Using \cite[Corollary 2.24]{Knapp2002} we have
\[ 0 = \langle H^{L}, H_{L} \rangle = \sum_{\alpha \in \Sigma} \alpha(H^{L}) \alpha(H_{L}) \,, \]
so that there is at least one strictly positive term in this sum. Let $\alpha \in \Sigma$ be such that $\alpha(H^{L}) \alpha(H_{L}) > 0$. Take now a nonzero $X \in \mathfrak{k} \cap (\mathfrak{g}_{\alpha} + \mathfrak{g}_{- \alpha})$, then \eqref{secondtermnonemptyproofroots} is strictly negative. That is, the second term in \eqref{existenceHessiansimplified} is strictly negative.

Observe that when $Y \in \mathfrak{k} \cap \mathfrak{l}$, adding $Y$ to $X$ does not change the first two terms of \eqref{existenceHessiansimplified}, because $[Y, H_{L}] = 0$. Thus in order to make the third term in \eqref{existenceHessiansimplified} zero and obtain a contradiction, it suffices to find $Y \in \mathfrak{k} \cap \mathfrak{l}$ with
\[ E_{\mathfrak{a}}([Y, \operatorname{Ad}_{k^{-1}}(H^{L})]) = E_{\mathfrak{a}}([X, \operatorname{Ad}_{k^{-1}}(H^{L})]) \,, \]
and it will follow that
\[ (\operatorname{Hess}_{k} g_{G, H_0}) (X - Y, X - Y) = -2 \left\lVert [X, H_{L}] \right\rVert^{2} - 2 \langle [X, \operatorname{Ad}_{k^{-1}}(H^{L})], [X, H_{L}] \rangle < 0 \,. \]
Let $L' \subset L$ be the smallest semistandard Levi subgroup with the property that $k \in L' M'$, say $k  = \ell'm$. From $k \in \mathcal{C}(L, H^{L})$ it follows that $\operatorname{Ad}_{k^{-1}} (H^{L}) \perp \mathfrak{a}$, and therefore $H^{L} \in \mathfrak{a}^{L'}$. We also have that $[\mathfrak{k}, \operatorname{Ad}_{k^{-1}}(H^{L})] \perp \operatorname{Ad}_{k^{-1}}(\mathfrak{a}) \supset \operatorname{Ad}_{m}^{-1}(\mathfrak{a}_{L'})$, so that $E_{\mathfrak{a}}([\mathfrak{k}, \operatorname{Ad}_{k^{-1}}(H^{L})]) \subset \operatorname{Ad}_{m}^{-1}(\mathfrak{a}^{L'})$. Hence it suffices to show that the map
\begin{align} \label{nonemptycmapsurjective}
\begin{split}
\mathfrak{k} \cap \mathfrak{l} & \to \operatorname{Ad}_{m}^{-1}(\mathfrak{a}^{L'}) \\
Y & \mapsto E_{\mathfrak{a}}([Y, \operatorname{Ad}_{k^{-1}}(H^{L})])
\end{split}
\end{align}
is surjective. Its restriction to $\mathfrak{k} \cap \mathfrak{l}'$ is precisely $-\operatorname{Ad}_{m}^{-1} \circ (Df_{L', H^{L}})_{\ell'}$ (compare \eqref{differentialf}). We seek to apply Lemma~\ref{criticalsetsubmersion} to $L'$.

By assumption we have $H^{L} \in \mathfrak{a}^{\operatorname{reg}}$, so that $H^{L} \in (\mathfrak{a}^{L'})^{\operatorname{reg}}$. By minimality of $L'$, we have $\ell' \notin \bigcup_{L'' \subsetneq L'} M' L''$. Therefore Lemma~\ref{criticalsetsubmersion} shows that $(Df_{L', H^{L}})_{\ell'}$ is surjective, so that the map \eqref{nonemptycmapsurjective} is surjective and the conclusion follows.
\end{proof}

\begin{corollary} \label{Cnonempty}When $H_{0} \in \mathfrak{a}^{\operatorname{gen}}$ we have $\mathcal{C}(G, H_{0}) \neq \varnothing$.\end{corollary}

\begin{proof}
The continuous function $g_{G, H_{0}} : K \to \mathbb{R}_{\geq 0}$ attains a minimum, which must be $0$ by Lemma~\ref{criticalpointgminimum}.
\end{proof}

\subsection{Uniqueness of critical points}

\label{sectionuniqueness}

\begin{lemma} \label{hessianquadformnegativedefinite} Let $H_{0} \in \mathfrak{a}^{+}$ and $k \in K - \bigcup_{L\in \mathcal L - \{ G\}} M' L$. Then the quadratic form on $\mathfrak{a}$ defined by
\[ H \mapsto \langle H_0, [ \operatorname{Ad}_{k}(H), E_{\mathfrak n} ( \operatorname{Ad}_{k}(H) ) ] \rangle \]
is negative definite. In particular, it is nondegenerate.
\end{lemma}

\begin{proof}Take a nonzero $H \in \mathfrak{a}$, and write the element $\operatorname{Ad}_{k}(H) \in \mathfrak{p}$ in the restricted root space decomposition \eqref{rootspacedecomp} as $\sum_{\alpha \in \Sigma \cup \{ 0 \} } X_{\alpha}$. Then $\theta X_{-\alpha} = -X_\alpha$. By Lemma~\ref{computationAsecondderivative}, the quadratic form evaluated at $H$ equals
\[ \sum_{\alpha \in \Sigma^{+}} \langle -X_{\alpha}, 2 X_{\alpha} \rangle_{\theta} \langle H_\alpha, H_0 \rangle = -2 \sum_{\alpha \in \Sigma^{+}} \lVert X_{\alpha} \rVert_{\theta}^2 \cdot \alpha(H_0) \,, \]
which is nonpositive because $H_{0}$ lies in the positive Weyl chamber. Therefore the quadratic form is negative semidefinite. The fact that $k \notin \bigcup_{L \in \mathcal L - \{ G\}} M' L$ implies $\operatorname{Ad}_{k}(H) \notin \mathfrak{a}$ (Lemma~\ref{conjugatetoAimplieslevi}). So at least one $X_{\alpha}$ is nonzero, and the statement follows.
\end{proof}

\begin{remark}\label{hessiancriticalsometimesdef}When $H_{0}$ lies in a mixed Weyl chamber, the quadratic form in Lemma~\ref{hessianquadformnegativedefinite} can sometimes be degenerate. This happens already for $G = \operatorname{SL}_{3}(\mathbb{R})$, even for completely generic values of $H_0$, and there appears to be no structure in the bad pairs $(H_0, k)$.\end{remark}

\begin{proposition} \label{surjectivityorbitalintegralH} Let $H_{0} \in \mathfrak{a}^{+}$ and $k \in K - \bigcup_{L \in \mathcal L - \{ G\}} M'L$. Then the Hessian of $h_{H_{0}, k}$ is everywhere negative definite.\end{proposition}

\begin{proof}Take $a \in A$ and define $k_{1} = \kappa(k a)$. From Lemma~\ref{kpartofmL} it follows that $k_{1} \notin \bigcup_{L\in \mathcal L - \{ G\}} M' L$. Using Lemma~\ref{computationAsecondderivative} we have for $H \in \mathfrak{a}$ that
\begin{align*}
(\operatorname{Hess}_{a} h_{H_{0}, k})(H) & = \langle H_{0}, [\operatorname{Ad}_{k_{1}}(H), E_{\mathfrak{n}}(\operatorname{Ad}_{k_{1}}(H))] \rangle \\
& = \langle [H_{0}, \operatorname{Ad}_{k_{1}}(H)], E_{\mathfrak n}(\operatorname{Ad}_{k_{1}}(H)) \rangle \,.
\end{align*}
By Lemma~\ref{hessianquadformnegativedefinite}, this quadratic form is negative definite.
\end{proof}

\begin{corollary} \label{criticalpointshnondegenerate}When $H_{0} \in \mathfrak{a}^{\operatorname{gen}, +}$ and $k \in K$, the critical points of $h_{H_{0}, k}$ are nondegenerate.\end{corollary}

\begin{proof}If $h_{H_{0}, k}$ has a critical point, then $k \notin \bigcup_{L \in \mathcal L - \{ G\}} M'L$ by Proposition~\ref{mLnocriticalpoints}. The Hessian of $h_{H_{0}, k}$ is then everywhere nondegenerate by Proposition~\ref{surjectivityorbitalintegralH}. In particular, its critical points are nondegenerate.
\end{proof}

\begin{proposition} \label{uniquenesscriticalpointcartan}When $H_{0} \in \mathfrak{a}^{\operatorname{gen}, +}$ and $k \in K$, the function $h_{H_{0}, k}$ has at most one critical point, which maximizes $h_{H_{0}, k}$.\end{proposition}

\begin{proof}By Proposition~\ref{mLnocriticalpoints}, the existence of one critical point implies that $k \notin \bigcup_{L \in \mathcal L - \{ G\}} M'L$. Proposition~\ref{surjectivityorbitalintegralH} then implies that the Hessian of $h_{H_{0}, k}$ is everywhere negative definite, so that there can be no other critical point and any critical point maximizes $h_{H_{0}, k}$.
\end{proof}

\subsection{Structure of the level sets}

\label{sectionstructurelevelsets}

We have all the necessary information to describe the set $\mathcal C(G, H_0)$.

\begin{proposition} \label{genericcriticalsetstructuresemisimple} When $H_{0} \in \mathfrak{a}^{\operatorname{gen}}$, the set $\mathcal{C}(G, H_{0})$ is a smooth manifold of dimension $\dim K - \dim A$ that varies smoothly with $H_{0}$.
\end{proposition}

\begin{proof} By Corollary~\ref{Cnonempty}, this set is nonempty, and by Lemma~\ref{maintermgenericsubmersionatcriticalsetreductive}, $f_{G, H_{0}}$ is a submersion at the points of $\mathcal{C}(G, H_{0})$. Thus $\mathcal{C}(G, H_{0})$ is a smooth manifold of dimension $\dim K - \dim A$, and from the local normal form for submersions it follows that it admits local parametrizations that depend smoothly on $H_{0}$.
\end{proof}

\begin{lemma}\label{CinvariantbyM} When $H_0 \in \mathfrak a$, the set $\mathcal{C}(G, H_{0})$ is invariant on both sides by $M$.
\end{lemma}

\begin{proof}This is clear from the definition \eqref{definitionCreductive}, using $\operatorname{Ad}_G$-invariance of the Killing form.
\end{proof}

\begin{remark}Many arguments simplify when $\mathcal C(G, H_0)$ is but a finite union of cosets of $M$ in $K$, effectively reducing arguments to the case of $G = \operatorname{SL}_2(\mathbb R)$. However, the set $\mathcal C(G, H_0)$ is usually much bigger; see Proposition~\ref{Klargecharacterization}.\end{remark}

\begin{corollary}\label{levelsetGcodim}When $H_{0} \in \mathfrak{a}^{\operatorname{gen}}$ and $a \in A$, the set of $g \in G$ for which $a$ is a critical point of $h_{H_0, g}$ is a smooth submanifold of codimension $\dim(A)$. It is stable on the left by $NAM$ and on the right by $M$.\end{corollary}

\begin{proof}From \eqref{heightfunctionreducetokpart} we saw that $h_{H_0, g}$ and $h_{H_0, \kappa(g)}$ have the same critical points. Therefore by Lemma~\ref{levelsetszerocritpoint} the set in question is equal to $NA \cdot \kappa(\mathcal C(G, H_0)a^{-1})$, which has codimension $\dim(A)$ in $G = NAK$ by Proposition~\ref{genericcriticalsetstructuresemisimple}. It is stable on the left by $NA$. It is stable on the left by $M$ because $M$ normalizes $NA$, which implies that $\kappa(\cdot)$ commutes with left multiplication by $M$, and $\mathcal C(G, H_0)$ is stable on the left by $M$ by Lemma~\ref{CinvariantbyM}. Similarly, invariance on the right follows from Lemma~\ref{CinvariantbyM}, using that $\kappa(\cdot)$ commutes with right multiplication by $K$.
\end{proof}

When $H_0 \in \mathfrak a$, define $\mathcal{R}_{H_{0}} \subset G$ to be the set of elements $g$ for which the function $h_{H_{0}, g}$ has a critical point.

\begin{proposition} \label{regularsetGopen} When $H_{0} \in \mathfrak{a}^{\operatorname{gen}, +}$, the set $\mathcal{R}_{H_0} \subset G$ is open and stable on the left by $NAM$.\end{proposition}

\begin{proof} Take $g_0 \in \mathcal{R}_{H_0}$. By Proposition~\ref{criticalpointshnondegenerate}, the critical point of $h_{H_{0}, g_0}$ is nondegenerate, call it $\xi$. We may reformulate this by saying that the map
\begin{align*}
A & \to \mathfrak{a}^{*} \\
a & \mapsto (D h_{H_{0}, g_0})_{a}
\end{align*}
has invertible differential at the level set above $0$, which is the singleton $\{\xi \}$. By the implicit function theorem applied to this smooth map with parameter $g \in G$, it follows that it has a zero for all $g$ in a neighborhood of $g_0$, which is to say that $\mathcal{R}_{H_0}$ is open in $G$. The stability under $NAM$ follows from Corollary~\ref{levelsetGcodim}.
\end{proof}

\begin{remark} \label{atonegativechamber} Let $\rho : G \to \operatorname{SL}(V)$ be any representation with finite kernel, and let $\Omega_{\mathcal S^0} \subset G$ be the set defined in \S\ref{npartcanonicalpartitions}. It is open and dense by Proposition~\ref{omegaopendense}. It is reasonable to expect that for $g \in \Omega_{\mathcal S^0}$ and $\lambda \in (\mathfrak a^*)^+$ we have $\lambda(H(ga)) \to - \infty$ as $a \to \infty$ in $A$. This would then imply that when $H_0 \in \mathfrak a^+$ the function $h_{H_0, g}$ has a critical point for all $g \in \Omega_{S^0}$, in particular, for all $g$ in an open dense subset of $G$ that does not depend on $H_0$. In fact, when $H_0 \in \mathfrak a^{\operatorname{gen}, +}$ it is reasonable to expect that $\mathcal R_{H_0} = \Omega_{\mathcal S^0}$.

The statement about the limit and the corollary that $\mathcal R_{H_0} \supset \Omega_{\mathcal S^0}$ are certainly true when $G = \operatorname{SL}_n(\mathbb R)$ and $\rho = \operatorname{Std}$. Then the entries of $H(ga)$ can be expressed in terms of quotients of sums of subdeterminants (using \eqref{normprojwedgesingle}), and when $g \in \Omega_{\mathcal S^0}$ those entries are bounded from above and below by constants times the entries of $a$, but arranged in decreasing order. The general case would require a more careful analysis of the weights of the exterior powers of $V$. This would in particular give an alternative proof of Corollary~\ref{Cnonempty}, whose only proof we have now is very technical and little intuitive.\end{remark}

\subsection{Some dimension bounds}

\begin{lemma}\label{complexrootsbigger}Let $\mathfrak g$ be a complex simple Lie algebra of rank $r$, $\mathfrak h$ a Cartan subalgebra with roots $\Sigma$ and a choice of simple roots $\Pi$. Let $\alpha \in \Pi$. Then there exist at least $r$ roots $\beta \in \Sigma$ with the property that $\beta \geq \alpha$.\end{lemma}

\begin{proof}We will show by induction the following statement: for every $1 \leq m \leq r$, there exists a set of simple roots $S$ with $|S| = m$ whose span contains $m$ linearly independent roots $\beta \geq \alpha$. For $m = 1$ we take $S = \{\alpha\}$. Assume we have found $S$ of cardinality $m < r$. Because $\mathfrak g$ is simple, $\Sigma$ is irreducible, so there exists a simple root $\gamma \in \Pi - S$ which is not orthogonal to all roots in $S$. Thus there exists a positive root $\beta \in \operatorname{span}(S)$ with $\beta \geq \alpha$ which is not orthogonal to $\gamma$. Because $\langle \beta, \gamma \rangle_\theta \leq 0$ by \cite[Lemma 2.51]{Knapp2002}, we must have $\langle \beta, \gamma \rangle_\theta < 0$. By \cite[Proposition 2.48]{Knapp2002} this implies that $\beta + \gamma$ is a root. We may now take $S' = S \cup \{\gamma\}$, whose span contains $m$ linearly independent roots $\geq \alpha$ in $\operatorname{span}(S)$ together with the root $\beta + \gamma \geq \alpha$. This completes the induction.
\end{proof}

\begin{lemma}\label{simplewithfewroots}Let $\mathfrak g$ be a real simple Lie algebra with Cartan decomposition $\mathfrak k \oplus \mathfrak p$, maximal abelian subalgebra $\mathfrak a \subset \mathfrak p$ with restricted roots $\Sigma \subset \mathfrak a^*$ and choice of positive roots $\Sigma^+$. The following are equivalent:
\begin{enumerate}
\item $\mathfrak g$ is compact or isomorphic to $\mathfrak{sl}_2(\mathbb R)$.
\item $\sum_{\alpha \in \Sigma^+} m(\alpha) = \dim(\mathfrak a)$.
\end{enumerate}
\end{lemma}

\begin{proof}Write $r = \dim(\mathfrak a)$.
Note that $\Sigma$ spans $\mathfrak a^*$ so that we always have
\[ \sum_{\alpha \in \Sigma^+} m(\alpha) \geq \#\Sigma^+ \geq r \,. \]
It is clear that equality holds if $\mathfrak g$ is compact or isomorphic to $\mathfrak{sl}_2(\mathbb R)$.

Now let $\mathfrak g$ be any simple Lie algebra. We will show that equality only holds for the examples stated. There are two cases, coming from the classification of simple real Lie algebras.

Suppose first that the complexification $\mathfrak g_{\mathbb C}$ is not simple. By \cite[Theorem 6.94]{Knapp2002} $\mathfrak g$ is the restriction of scalars of a simple complex Lie algebra. Let $\mathfrak h \supset \mathfrak a$ be a Cartan subalgebra. It follows as in \cite[\S VI.11, p.425]{Knapp2002} that all roots of $\mathfrak h$ have nonzero restriction to $\mathfrak a$ (and all restrictions are different) and all restricted roots have multiplicity $2$. It follows that
\[ \sum_{\alpha \in \Sigma^+} m(\alpha) = 2\#\Sigma^+ \geq 2r \geq r + 1 \,, \]
because $r \geq 1$, meaning that equality does not hold in this case.

Suppose now that the complexification $\mathfrak g_{\mathbb C}$ is simple. Let $\Pi$ be a system of simple restricted roots of $\mathfrak a$. Let $\theta$ be a Cartan involution corresponding to $\mathfrak k \oplus \mathfrak p$ and $\mathfrak h \subset \mathfrak g$ a $\theta$-stable Cartan subalgebra containing $\mathfrak a$. We may choose an ordering on $\mathfrak h_{\mathbb C}^*$ such that the restrictions of positive roots of $\mathfrak h_{\mathbb C}$ to $\mathfrak a$ are either zero or positive. As in \cite[\S VI.12, Problem 7]{Knapp2002}, all simple $\alpha \in \Pi$ are restrictions of simple roots of $\mathfrak h_{\mathbb C}$. Let $\alpha_1, \ldots, \alpha_r$ be simple lifts to $\mathfrak h_{\mathbb C}$ of the simple restricted roots of $\mathfrak a$. If $\mathfrak g$ is noncompact, then $r \geq 1$. Because $\mathfrak g_{\mathbb C}$ is simple, by Lemma~\ref{complexrootsbigger} there are at least $\dim_{\mathbb C}(\mathfrak h_{\mathbb C}) - 1$ positive roots $\beta$ of $\mathfrak h_{\mathbb C}$ with $\beta > \alpha_1$. The restrictions of all these roots are positive. Though the restrictions may coincide, they give the following lower bound for multiplicities:
\[ \sum_{\alpha \in \Sigma^+} m(\alpha) \geq r + \dim_{\mathbb C}(\mathfrak h_{\mathbb C}) - 1 \,. \]
if this is equal to $r$, then $\dim_{\mathbb C}(\mathfrak h_{\mathbb C}) = 1$. That is, $\mathfrak g_{\mathbb C} \cong \mathfrak{sl}_2(\mathbb C)$. Because we are assuming $\mathfrak g$ is noncompact, it must be the unique noncompact form of $\mathfrak{sl}_2(\mathbb C)$, which is $\mathfrak{sl}_2(\mathbb R)$.
\end{proof}

\begin{proposition}\label{Klargecharacterization}Let $G$ be a semisimple Lie group. The following are equivalent:
\begin{enumerate}
\item $\dim(K) - \dim(A) = \dim(M)$.
\item All simple factors of $\mathfrak g$ are either compact or $\mathfrak{sl}_2(\mathbb R)$.
\end{enumerate}
\end{proposition}

Note that we always have $\dim(K) \geq \dim(M) + \dim(A)$.

\begin{proof}
The difference $\dim(K) - \dim(M)$ is equal to $\sum_{\alpha \in \Sigma^+} m(\alpha)$. If this is an equality, then we must have the same equality for all simple factors of $\mathfrak g$. By Lemma~\ref{simplewithfewroots}, this is equivalent to saying that  all simple factors of $\mathfrak g$ are either compact or $\mathfrak{sl}_2(\mathbb R)$.
\end{proof}

\subsection{Proofs specific to \texorpdfstring{$\operatorname{SL}_3(\mathbb{R})$}{SL3(R)}}

Some of the lemmas that go into the proof of Proposition~\ref{genericcriticalsetstructuresemisimple} admit more direct and computational proofs when $G = \operatorname{SL}_{3}(\mathbb{R})$. We include one of those here, because it features a rather curious inequality.

\begin{proof}[Direct proof of Corollary~\ref{Cnonempty} when $G = \operatorname{SL}(3, \mathbb{R})$] We use the standard choice of Iwasawa decomposition. Write $H_{0} = \operatorname{diag}(a, b, c)$. Proving that the set $\mathcal{C}(G, H_{0})$ is nonempty is equivalent to showing that there exists a symmetric matrix $X = \operatorname{Ad}_{k^{-1}}(H_{0})$ with zeros on the diagonal, which is isospectral with $H_{0}$. Write $X = \begin{psmallmatrix}
0 & x & y \\
x & 0 & z \\
y & z & 0
\end{psmallmatrix}$. Then $H_{0}$ and $X$ have the same characteristic polynomial when
\begin{align} \label{XHnoughtequations}
\begin{cases}
x^{2} + y^{2} + z^{2} = - (ab + bc + ca) = \frac{1}{2}(a^{2} + b^{2} + c^{2}) \\
2 xyz = abc \,.
\end{cases}
\end{align}
Because the first equation does not see the signs of $x, y, z$, squaring the second does not affect solvability. We may now view \eqref{XHnoughtequations} as prescribing the arithmetic and geometric mean of $x^{2}, y^{2}, z^{2}$, namely $\frac{1}{6}(a^{2} + b^{2} + c^{2})$ respectively $(\frac{1}{4} a^{2}b^{2}c^{2})^{1/3}$. It is well known that a necessary and sufficient condition for this to have a solution in nonnegative numbers $x^{2}, y^{2}, z^{2}$, is that the arithmetic mean is at least the geometric mean, meaning that
\[ \frac{1}{6}(a^{2} + b^{2} + c^{2}) \geq \left( \frac{1}{4} a^{2}b^{2}c^{2} \right)^{1/3} \,. \]
That this condition is satisfied is Lemma~\ref{AMGMlemma} below, and it does not require the hypothesis that $H_{0} \in \mathfrak{a}^{\operatorname{gen}}$.
\end{proof}

\begin{lemma}[AM--\texorpdfstring{$\sqrt[3]{2}$}{sqrt 3 2}GM] \label{AMGMlemma}Let $a, b, c \in \mathbb{R}$ with $a + b + c = 0$. Then
\[ \left( \frac{a^{2} + b^{2} + c^{2}}{3} \right)^{3} \geq 2 \cdot a^{2}b^{2}c^{2} \,. \]
\end{lemma}

The proof of Lemma~\ref{AMGMlemma} is an amusing exercise using the AM--GM inequality. Equality holds if and only if $(a, b, c)$ is proportional to $(1, 1, -2)$ or a permutation thereof. With the notations as in Corollary~\ref{Cnonempty}, this corresponds to the case where $H_{0}$ lies in $\mathfrak{a}_{L}$ for some semistandard Levi subgroup $L \in \mathcal L - \{ G\}$.

\section{The \texorpdfstring{$K$}{K}-projection and the geodesic flow}
\label{secKprojection}

In this subsection we prove Theorem~\ref{KparttoMcentralizer}. For the possibility of uniformity, see Remark~\ref{remarkKvsNproofs}. For $H \in \mathfrak{a}$ and $k \in K$, define a map
\begin{align}
\label{definitionp}
\begin{split}
p_{H, k} : \mathbb{R} & \to \mathfrak{p} \\
t & \mapsto \operatorname{Ad}_{\kappa(k e^{tH})}(H) \,.
\end{split}
\end{align}
Using \eqref{derivativeAd} and Lemma~\ref{differentialkpart} to differentiate $\kappa$, we find that
\[ p_{H, k}' (t) = [E_{\mathfrak{k}}(p_{H, k}(t)), p_{H, k}(t)] \,.\]
Therefore $p_{H, k}(t)$ is a solution to the homogeneous quadratic differential equation in $\mathfrak{p}$ given by
\begin{equation} \label{diffequationpt}
X' = [E_{\mathfrak{k}}X, X] 
\end{equation}
with initial value $\operatorname{Ad}_{k}(H)$. Its flow
\begin{align} \label{dynamicalsystem}
\begin{split}
p : \mathbb{R} \times \mathfrak{p} & \to \mathfrak{p} \\
(t, X) & \mapsto p_{t}(X) \,.
\end{split}
\end{align}
is given explicitly by $p_{t}(X) = p_{H, k}(t)$ when $X = \operatorname{Ad}_{k}(H)$, and is in particular defined for all $t \in \mathbb{R}$. We will prove Theorem~\ref{KparttoMcentralizer} by gathering enough information about the dynamical system \eqref{dynamicalsystem}.

\begin{example}Let $G = \operatorname{PSL}_{2}(\mathbb{R})$ and identify $\mathfrak{p}$ with $\mathbb{R}^{2}$ via an isometry that sends $\mathfrak{a}$ to the horizontal axis $\mathbb{R} \times \{ 0 \}$. It is clear from \eqref{diffequationpt} that the points of $\mathfrak{a}$ are stationary, and it is clear from \eqref{definitionp} that the norm of $X \in \mathfrak{p}$ is invariant under the flow. It is then not hard to see that the phase portrait of the dynamical system \eqref{dynamicalsystem} is as follows: the points of $\mathfrak{a}^{+}$ are unstable equilibra, the points of $\mathfrak{a}^{-}$ are stable equilibria, and apart from the equilibrium at $0$ every other orbit is heteroclinic and describes a Euclidean half circle with endpoints on the horizontal axis, starting at $\mathfrak{a}^{+}$ and ending at $\mathfrak{a}^{-}$.\end{example}

It is clear from \eqref{diffequationpt} that all elements of $\mathfrak{a}$ are equilibria. The key in the proof of Theorem~\ref{KparttoMcentralizer} is the existence of functions that are monotonic along orbits. For a root $\alpha \in \Sigma$, let $H_{\alpha} \in \mathfrak{a}$ be as in \S\ref{secroots}. The set $\{ H_{\alpha} : \alpha \in \Pi \}$ is a basis of $\mathfrak{a}$. When $\alpha \in \Pi$ and $H \in \mathfrak{a}$, define $c_{\alpha}(H)$ to be the $\alpha$th coordinate of $H$ in this basis. More generally, define $c_{\alpha}(X) = c_{\alpha}(E_{\mathfrak{a}}(X))$ for $X \in \mathfrak{g}$.

\begin{lemma} \label{aptderivative}Let $X \in \mathfrak{p} - \mathfrak{a}$. Then $c_{\alpha}([E_{\mathfrak{k}}X, X]) \leq 0$ for all $\alpha \in \Pi$, and at least one is strictly negative. That is, $E_{\mathfrak{a}}[E_{\mathfrak{k}}X, X]$ is nonzero and lies in the closure of the negative Weyl chamber.\end{lemma}

\begin{proof}Write $X = \sum_{\alpha \in \Sigma \cup \{ 0 \}} X_{\alpha}$ in the restricted root space decomposition \eqref{rootspacedecomp}. Using that $X_{- \alpha} = - \theta(X_{\alpha})$ and $E_{\mathfrak{k}} X = \sum_{\alpha \in \Sigma^{+}} (X_{- \alpha} - X_{\alpha})$, we find that
\begin{align*}
E_{\mathfrak{a}}[E_{\mathfrak{k}}X, X] & = \sum_{\alpha \in \Sigma^{+}} [X_{- \alpha} - X_{\alpha}, X_{- \alpha} + X_{\alpha}] \\
& = - 2 \sum_{\alpha \in \Sigma^{+}} [X_{\alpha}, X_{- \alpha}] \\
& = -2 \sum_{\alpha \in \Sigma^{+}} \langle X_{\alpha}, X_{- \alpha} \rangle \cdot H_{\alpha} \\
& =  -2 \sum_{\alpha \in \Sigma^{+}} \lVert X_{\alpha} \rVert_{\theta}^{2} \cdot H_{\alpha} \\
& = -2 \sum_{\alpha \in \Pi} \sum_{\beta \geq \alpha} \lVert X_{\beta} \rVert_{\theta}^{2} \cdot H_{\alpha}
\end{align*}
where in the third equality we have used \cite[\S II, Lemma 2.18]{Knapp2002}. Therefore the coefficients of $E_{\mathfrak{a}}[E_{\mathfrak{k}}X, X]$ in the basis $\{ H_{\alpha} : \alpha \in \Pi \}$ are nonnegative, and because $X \notin \mathfrak{a}$ at least one is nonzero. (This is essentially the same proof as that of Lemma~\ref{hessianquadformnegativedefinite}.)
\end{proof}

\begin{lemma} \label{apartdecreasing} Let $X \in \mathfrak{p} - \mathfrak{a}$ and $t > 0$. Then $c_{\alpha}(p_{t}(X)) \leq c_{\alpha}(X)$ for all $\alpha \in \Pi$, and at least one inequality is strict.
\end{lemma}

\begin{proof}By Lemma~\ref{aptderivative}, every $c_{\alpha}(p_{t}(X))$ is monotonically decreasing, and near every $t \in \mathbb{R}$ at least one is strictly decreasing.
\end{proof}

\begin{lemma} \label{nonwanderinga}The only non-wandering points for the dynamical system \eqref{dynamicalsystem} are those of $\mathfrak{a}$.
\end{lemma}

\begin{proof}Let $X \in \mathfrak{p} - \mathfrak{a}$. By Lemma~\ref{apartdecreasing} there exists $\alpha \in \Pi$ for which $c_{\alpha}(p_{1}(X)) < c_{\alpha}(X)$. By continuity of the flow, there exists a neighborhood $U$ of $X$ such that $c_{\alpha}(p_{1}(U)) < c_{\alpha}(U)$. Because $c_{\alpha}$ is decreasing along orbits, we have $p_{t}(U) \cap U = \varnothing$ for all $t \geq 1$. Thus $X$ is wandering.
\end{proof}

\begin{lemma} \label{dynamicallimita}Let $X \in \mathfrak{p}$. There exists $H \in \mathfrak{a}$ such that $\lim_{t \to \infty }p_{t}(X) = H$. \end{lemma}

\begin{proof}Write $X = \operatorname{Ad}_{k}(H)$ with $k \in K$ and $H \in \mathfrak{a}$. Consider its $\omega$-limit $\omega(X)$. By Lemma~\ref{nonwanderinga} we have $\omega(X) \subset \mathfrak{a}$. In view of the explicit expression for $p_{t}(X)$, we have $\omega(X) \subset \operatorname{Ad}_{K}(H)$ by continuity. Thus $\omega(X) \subset \mathfrak{a} \cap \operatorname{Ad}_{K}(H)$. By Lemma~\ref{conjugatetoAimplieslevi}, this intersection is equal to the Weyl group orbit of $H$. In particular, it is discrete. Because $\omega(X)$ is connected, it must consist of a single point of $\mathfrak{a}$.
\end{proof} 

\begin{proof}[Proof of Proposition~\ref{KparttoMcentralizer}] The $K$-projection of $ge^{tH}$ does not depend on the triangular part of $g$, so we may assume that $g \in K$. Conjugation by $k \in K$ gives a diffeomorphism
\[ K / Z_{K}(H) \to \operatorname{Ad}_{K}(H) \,. \]
Take $k \in G$. When $t \to + \infty$, the image of $\kappa(k e^{tH})$ under this map tends to an element $\operatorname{Ad}_{m}(H) \in \operatorname{Ad}_{M'}(H)$ by Lemma~\ref{dynamicallimita}. Therefore $\kappa(k e^{tH})$ tends to the point $m Z_{K}(H)$ in the quotient $K / Z_{K}(H)$. Equivalently, $\kappa(k e^{tH})$ tends to the set $m Z_{K}(H)$ in $K$.
\end{proof}

\begin{remark}\label{remarkKvsNproofs}In general the convergence in Lemma~\ref{dynamicallimita} is not uniform. Already for $G = \operatorname{SL}_3(\mathbb R)$ there are heteroclinic orbits that have nearby orbits with totally different limit points, even some that emanate from equilibria in the positive Weyl chamber. Some of these orbits, that come in one-dimensional families, can be seen to lie in Levi subalgebras (Lie algebras of $L \in \mathcal L$). But not all of them do and certainly not those that come in two-dimensional families. In fact, it seems likely that these badly behaved orbits correspond exactly to the partition of $K$ into the sets $\kappa(\Omega_{\mathcal S})$ defined in \S\ref{npartcanonicalpartitions}; compare Example~\ref{omegaslthree}.

Further evidence for this is the following observation. Let $\rho : G \to \operatorname{SL}(V)$ be any representation with finite kernel, and let $\Omega_{\mathcal S^0}$ be the set defined in \S\ref{npartcanonicalpartitions}. As explained in Remark~\ref{atonegativechamber} we expect that for $g \in \Omega_{\mathcal S^0}$ and $\lambda \in (\mathfrak a^*)^+$ we have $\lambda(H(ga)) \to - \infty$ as $a \to \infty$ in $A$, and we sketched a proof of this when $G =\operatorname{SL}_n(\mathbb R)$ and $\rho = \operatorname{Std}$. Moreover, when $\lambda \in \Sigma^+$ is merely a positive root, it is still reasonable to expect that $\lambda(H(ga)) \to - \infty$ as $a \to \infty$ in $A$ uniformly along regular directions; this is also apparent for $G =\operatorname{SL}_n(\mathbb R)$. Now writing $ga = n'a'k'$, we have for $H \in \mathfrak a$ that $\operatorname{Ad}_{k'}(H) = \operatorname{Ad}_{a'^{-1} n'^{-1}g}(H)$. When $g \in  \Omega_{\mathcal S^0}$ stays in a compact set, then so does $n'$, by Theorem~\ref{npartboundeduniform}, and if $\lambda(a') \to - \infty$ for all positive roots $\lambda$, then the positive root space components of $\operatorname{Ad}_{a'^{-1} n'^{-1}g}(H)$ tend to zero. On the other hand this equals $\operatorname{Ad}_{k'}(H)$, which lives in a bounded subset of $\mathfrak p$, so that the negative root space components approach zero as well, meaning that $\operatorname{Ad}_{k'}(H) \to \mathfrak a$ and therefore $k' \to M' Z_K(H)$. Taking $H \in \mathfrak a$ regular, we would obtain that $k' \to M'$, uniformly for $g$ in compact subsets of $\Omega_{\mathcal S^0}$ and $a\to \infty$ inside regular cones of $A$.
\end{remark}

\bibliographystyle{plain}
\bibliography{../../../library/_bib/bibliography}

\end{document}